\documentclass[final,1p,times]{elsarticle}%
\usepackage{amsmath,multirow}
\usepackage{amssymb}
\usepackage{amscd}
\usepackage{amsfonts}
\usepackage{enumerate}
\usepackage{mathrsfs}
\usepackage{xy}
\usepackage{slashed}
\usepackage{stmaryrd}
\usepackage{graphicx}
\usepackage{fancybox}
\usepackage{color}
\usepackage{exscale}
\usepackage{enumitem,array}%
\usepackage{amsfonts}
\usepackage{epstopdf}
\usepackage{color}%
\usepackage{subfigure}

\newtheorem{mytheo}{Theorem}[section]
\newtheorem{mydef}[mytheo]{Definition}
\newtheorem{cor}[mytheo]{Corollary}

\newtheorem{defn}[mytheo]{Definition}
\newtheorem{rem}[mytheo]{Remark}

\newcounter{remark}

\newcounter{problem}
\newenvironment{problem}{\refstepcounter{problem}\vspace{1.5ex}
{\noindent\bf Problem
\theproblem.}\hspace{0.3em}\parindent=0pt}{\vspace{1ex}}

\makeatletter
\def\@upcite#1#2{\textsuperscript{[{#1\if@tempswa , #2\fi}]}}
\makeatother
\newenvironment{proof}{\vspace{1ex}
{\it Proof. }\hspace{0.3em}}{\vspace{1ex}} \journal{}
\begin{document}
\begin{frontmatter}
\title{Cost-reduction implicit exponential Runge--Kutta methods for  highly oscillatory systems}

\author[author1]{Xianfa Hu}
\ead{zzxyhxf@163.com}

\author[author1]{Wansheng Wang}
\ead{w.s.wang@163.com}

\author[author2]{Bin Wang}
\ead{wangbinmaths@xjtu.edu.cn}

\author[author3]{Yonglei Fang\corref{cor1}}
\ead{ylfangmath@163.com}

\cortext[cor1]{Corresponding author}
\address[author1]{Department of Mathematics, Shanghai Normal University, Shanghai 200234, P.R.China}

\address[author2]{School of Mathematics and Statistics, Xi'an Jiaotong University, Xi'an 710049, Shannxi,
P.R.China}

\address[author3]{School of Mathematics and Statistics, Zaozhuang University,  Zaozhuang 277160, P.R. China}


\begin{abstract}
In this paper, two novel classes  of implicit exponential Runge--Kutta  (ERK) methods are studied for solving highly oscillatory  systems.  First of all,  we analyze the symplectic conditions of two kinds of  exponential integrators, and present a first-order symplectic  method.  In order to  solve highly oscillatory problems,  the highly accurate implicit  ERK integrators (up to order four) are formulated by comparing the Taylor expansions of numerical and exact solutions, it is shown that  the  order conditions of two new kinds of  exponential methods are identical to the order conditions of  classical Runge--Kutta (RK) methods. Moreover, we  investigate  the linear stability properties of  these exponential methods. Finally, numerical results not only present the long time  energy preservation of the first-order symplectic method, but also illustrate the  accuracy and efficiency of these formulated methods in comparison with standard ERK methods.
 \end{abstract}

\end{frontmatter}
 \vskip0.5cm
\pagestyle{myheadings} \thispagestyle{plain} \markboth{}
{\centerline{\small  X. Hu, W. Wang, B. Wang  and Y. Fang}}
\section{Introduction}
It is well known that classical Runge--Kutta methods have a wide range of applications in scientific computing. Especially, the symplectic methods can preserve the symplecticity of the original systems, and the symplectic or symmetric methods provide long time energy preservation  applied to a Hamiltonian system.
Symplectic algorithms for Hamiltonian systems
appeared in 1980s, and the earliest significant contribution to this
field were due to Feng Kang (see \cite{Feng1985,Feng1986}). It is
worth noting the earlier important work on symplectic integration
by J. M. Sanz-Serna, who first found and analyzed symplectic Runge-Kutta
schemes for Hamiltonian systems (see  \cite{Sanz-Serna1988BIT}).
Symplectic exponential Runge--Kutta methods for solving Hamiltonian systems
were proposed by Mei et al. \cite{Mei2017}, and  shown better performance than classical symplectic
Runge--Kutta schemes. However, the coefficients of these exponential integrators
are strongly dependent on  the computing or approximating the product of a matrix exponential function
with a vector. As a consequence, we try to design two  novel classes of  implicit exponential Runge--Kutta methods with  lower computational cost.

In this work, we consider  the first-order  initial value
problem
\begin{equation}\label{equ1}\left\{
\begin{array}{l}
y'(t)+My(t)=f(y(t)),\quad t\in[t_0,t_{\rm
end}],\cr\noalign{\vskip1truemm} y(t_0)=y_0,
\end{array}
\right.\end{equation} {where} the matrix  $M\in \mathbb R^{m\times
m}$ is  symmetric positive definite or skew-Hermitian with
eigenvalues of large modulus. Problems of the form (\ref{equ1})
arise frequently in a variety of applied science  such as quantum mechanics,
flexible mechanics, and electrodynamics. Some highly oscillatory
 problems (see, e.g. \cite{Hochbruck2010,Wang2017}), Schr\"{o}dinger equations (see, e.g. \cite{Brugnano2018,Celledoni2008,Wang2022}) and KdV equations (see, e.g. \cite{Wang2012})
 can be converted into (\ref{equ1}) with appropriate spatial discretizations.  The exact solution of
problem (\ref{equ1}) can be  represented by the Volterra integral formula
\begin{equation}\label{Volterra formula}
y(t_0+h)=e^{-hM}y(t_0)+h\int_0^1 e^{-(1-\tau)hM}f(y(t_0+h\tau))d\tau.
\end{equation}
It is a challenge for effectively solving  the problem (\ref{equ1}) once it has the stiffness matrix $M$. In formula (\ref{Volterra formula}),  $e^{-hM}$ is generally the matrix exponential function, and exponential integrators can exactly integrate the linear equation $y^{\prime}(t)+My(t)=0$, which indicates that exponential integrators have unique advantages for solving stiff or highly oscillatory problems than non-exponential integrators. Exponential integrators have been received more attention \cite{Fang2021,Hochbruck1997,Hochbruck1998,Hochbruck2005a,Hochbruck2005b,Hochbruck2010,Krogstad 2005,Lawson1967,Lawson1972,Wu2016,Mei2017,Wang2019}. It is also worth mentioning that extended Runge-Kutta-Nystr\"{o}m (ERKN) methods \cite{Fang2010a,Fang2010b,Wang2015,Wu2010,Wu2013,Wu2015,Wu2018}, as the exponential integrators, which were formulated for effectively solving second-order oscillatory systems.

On the other hand, (\ref{equ1}) frequently possesses some important geometrical or physical properties. When
$f(y)=J^{-1}\nabla U(y)$ and $-M=J^{-1}Q$, with the skew-symmetric
\begin{equation*}
J=\left(
\begin{array}{cc}
0&I\\
-I&0\\
\end{array}
\right),
\end{equation*}
where $U(y)$ is a smooth potential function, $Q$ is a symmetric matrix and $I$ is the identity matrix, the problem (\ref{equ1}) can be converted into a Hamiltonian system. Owing to this, our study starts by deriving the symplectic conditions of these methods. It is noted that standard exponential Runge-Kutta (ERK) methods based on the stiff-order conditions (comprise the classical order conditions) \cite{Hochbruck1998,Hochbruck2005a,Hochbruck2005b,Hochbruck2010}.  However, as claimed by Berland et al. in \cite{Berland2005}, the stiff order conditions are relatively strict.   In this paper, the coefficients of implicit exponential methods are real constant, therefore our study is related to the classical order, not satisfy the stiff order conditions.  We have presented that two new kinds of explicit ERK methods up to order four  reduce to  classical Runge-Kutta (RK) methods once $M\rightarrow \mathbf{0}$ \cite{Hu2022,Wang2022}. In what follows, we will study the implicit ERK methods.

The paper is organized as follows. In Section $\ref{sec2}$, we  investigate the symplectic conditions for the simplified version of ERK  (SVERK) and  modified version of ERK  (MVERK) methods respectively, and present  a first-order symplectic method. In Section $\ref{sec3}$,  the order conditions of implicit SVERK  and  MVERK  methods are derived, which are identical to the classical  order conditions  of RK methods, and  we present the implicit second-order  methods with one stage and  fourth-order  methods with two stages, respectively. Section $\ref{sec4}$ is devoted to the linear stability regions of implicit SVERK and MVERK methods.  In Section $\ref{sec5}$,   numerical experiments are carried out to show the structure-preserving property of the symplectic method and  present  the comparable  accuracy and  efficiency  of  these implicit ERK  methods. The last section is concerned with concluding remarks.

\section{The symplectic conditions for two new classes of  ERK methods}\label{sec2}

In our previous work, we have formulated the modified and simplified versions of explicit ERK methods for solving stiff or highly oscillatory problems, and presented the convergence of explicit exponential methods. Meanwhile, it has been pointed out that
the internal stages and update of  SVERK methods preserve some properties of matrix-valued functions, and
MVERK methods inherit the internal stages and modify the update of classical RK methods,
but their coefficients are independent of the entire functions $\varphi_k(-hM)$ \eqref{entire functions} of standard exponential integrators \cite{Hochbruck2010}.

\begin{mydef}\label{definition1} (\cite{Wang2022})
An $s$-stage SVERK  method for the numerical integration
(\ref{equ1}) is defined as
\begin{equation}\label{SVERK}
\left\{
\begin{array}{l}
\displaystyle Y_i=e^{-c_ihM}y_0+h\sum\limits_{j=1}^sa_{ij}f(Y_j), \quad
i=1,\ldots,s,\cr\noalign{\vskip4truemm}
\displaystyle y_{1}=e^{-hM}y_0+h\sum\limits_{i=1}^sb_{i}f(Y_i)+w_s(z),
\end{array}
\right.
\end{equation}
{where} $a_{ij} $, $b_i$
 are real constants for {$i,j=1,\ldots, s$},
 $Y_i\approx y(t_0+c_ih)$ for $i=1,\ldots,s$,  $w_s(z)$  depends on $h$, $M$,  and $w_s(z)\rightarrow 0$ when $M\rightarrow \mathbf{0}$. \end{mydef}

\begin{mydef}\label{definition2}(\cite{Wang2022})
An $s$-stage MVERK method for the numerical integration
(\ref{equ1}) is defined as
\begin{equation}\label{MVERK}
\left\{
\begin{array}{l}
\displaystyle \bar{Y}_i=y_0+h\sum\limits_{j=1}^s\bar{a}_{ij}(-M\bar{Y}_j+f(\bar{Y}_j)), \quad
i=1,\ldots,s,\cr\noalign{\vskip4truemm}
\displaystyle \bar{y}_{1}=e^{-hM}y_0+h\sum\limits_{i=1}^s\bar{b}_{i}f(\bar{Y}_i)+\bar {w}_s(z),
\end{array}
\right.
\end{equation}
{where} $\bar{a}_{ij}$, $\bar{b}_i$
 are real constants for ${i,j=1,\ldots, s}$, $\bar{Y}_i\approx y(t_0+\bar{c}_ih) $ for $i=1,\ldots,s$, $\bar{w}_s(z)$ is related to $h$ and $M$, and $\bar{w}_s(z)\rightarrow0$ once $M\rightarrow \mathbf{0}$.
\end{mydef}
The $w_s(z)$ and $\bar{w}_s(z)$  also depend on the
term $f(\cdot)$ and initial value $y_0$ once we consider the order of SVERK and  MVERK methods which satisfies $p\geq2$. If we consider the first-order methods, then  $w_s(z)=0$   and $\bar{w}_s(z)=0$.
It should be noted that  the SVERK or MVERK methods with the same order share the same $w_s(z)$ or $\bar{w}_s(z)$, and  $w_s(z)$ is different from $\bar{w}_s(z)$ when $p\geq3$ in \cite{Hu2022,Wang2022}.  It is clear  that SVERK and MVERK methods reduce to classical RK methods when $M\rightarrow\mathbf{0}$, and these methods  exactly integrate the first-order homogeneous linear system
\begin{equation}\label{homogeneous}
y'(t)=-My(t), \quad y(0)=y_0,
\end{equation}
with the exact solution
 $$y(t)=e^{-tM}y_0.$$

 The  SVERK method (\ref{SVERK}) can be displayed by the following Butcher Tableau
\begin{equation}\label{Tableau-SVERK}
\begin{aligned} &\quad\quad\begin{tabular}{c|c|c}
 ${c}$&$\mathbf{e^{-chM}}$&${A}$ \\
 \hline
  $\raisebox{-1.3ex}[1.0pt]{$e^{-hM}$}$ & $\raisebox{-1.3ex}[1.0pt]{$w_s(z)$}$&$\raisebox{-1.3ex}[1.0pt]{${b}^{\intercal}$}$  \\
\end{tabular}
~=
\begin{tabular}{c|c|ccc}
 ${c}_1$&$e^{-c_1hM}$&${a}_{11}$&$\cdots$&${a}_{1s}$\\
$\vdots$& $\vdots$ & $\vdots$&$\vdots$&$\vdots$\\
 ${c}_s$ &$e^{-c_shM}$&  ${a}_{s1}$& $\cdots$& ${a}_{ss}$\\
 \hline
 $\raisebox{-1.3ex}[1.0pt]{$e^{-hM}$}$&$\raisebox{-1.3ex}[1.0pt]{$w_s(z)$}$&$\raisebox{-1.3ex}[1.0pt]{${b}_1$}$&\raisebox{-1.3ex}[1.0pt]{$\cdots$} &  $\raisebox{-1.3ex}[1.0pt]{${b}_s$}$\\
\end{tabular}.
\end{aligned}
\end{equation}
where $c_i=\sum\limits_{j=1}^sa_{ij}$  for $i=1,\ldots,s$. Similarly, the MVERK method (\ref{MVERK})  also can be expressed in the Butcher tableau
\begin{equation}\label{Tableau-MVERK}
\begin{aligned} &\quad\quad\begin{tabular}{c|c|c}
 $\bar{c}$&$\mathbf{I}$&$\bar{A}$ \\
 \hline
  $\raisebox{-1.3ex}[1.0pt]{$e^{-hM}$}$ & $\raisebox{-1.3ex}[1.0pt]{$\bar{w}_s(z)$}$&$\raisebox{-1.3ex}[1.0pt]{${b}^{\intercal}$}$  \\
\end{tabular}
~=
\begin{tabular}{c|c|ccc}
 $\bar{c}_1$&$I$&$\bar{a}_{11}$&$\cdots$&$\bar{a}_{1s}$\\
$\vdots$& $\vdots$ & $\vdots$&$\vdots$&$\vdots$\\
 $\bar{c}_s$ &$I$&  $\bar{a}_{s1}$& $\cdots$& $\bar{a}_{ss}$\\
 \hline
 $\raisebox{-1.3ex}[1.0pt]{$e^{-hM}$}$&$\raisebox{-1.3ex}[1.0pt]{$\bar{w}_s(z)$}$&$\raisebox{-1.3ex}[1.0pt]{$\bar{b}_1$}$&\raisebox{-1.3ex}[1.0pt]{$\cdots$} &  $\raisebox{-1.3ex}[1.0pt]{$\bar{b}_s$}$\\
\end{tabular}.
\end{aligned}
\end{equation}
with $\bar{c}_i=\sum\limits_{j=1}^s \bar{a}_{ij}$  for $i=1,\ldots,s$.

It is true that  (\ref{equ1}) becomes a Hamiltonian system when $f(y)=J^{-1}\nabla U(y)$ and $M=-J^{-1}Q$, where $U(y)$ is a smooth potential function and $Q$ is a symmetric matrix. Thus, we consider the following Hamiltonian system
\begin{equation}\label{Hamiltonian}\left\{
\begin{array}{l}
y'(t)-J^{-1}Qy(t)=J^{-1}\nabla U(y),\quad t\in[t_0,t_{\rm
end}],\cr\noalign{\vskip1truemm}
y(t_0)=y_0.
\end{array}
\right.\end{equation}
Under the assumptions of $w_s(z)=0$ and $\bar{w}_s(z)=0$, we will analyze the symplectic conditions for SVERK and MVERK methods. In fact, $w_s(z)=0$ and $\bar{w}_s(z)=0$  mean that the order  of   SVERK and MVERK methods satisfies  $p\leq1$. Since $w_s(z)\neq 0$ and $\bar{w}_s(z)\neq 0$,
 the symplectic conditions of the SVERK  and MVERK methods  can be discussed in the same way, however, the $w_s(z)$
  and $\bar{w}_s(z)$ will break the  formal conservation of the symplectic invariant.

\begin{mytheo}\label{theorem1}
If the coefficients of an $s$-stage SVERK method with $w_s(z)=0$, which satisfy the following conditions:
\begin{equation}\label{symplectic conditions}\left\{
\begin{array}{l}
e^{-(1-c_k)hM}G^{-1}_{k}=G^{-1}_{k}(e^{-(1-c_k)hM})^{\intercal}, \quad  k=1,\ldots,s,\cr\noalign{\vskip1truemm}
b_{k}b_{l}=b_{k}a_{kl}Je^{-(1-c_k)hM}J^{-1}+b_{l}a_{lk}(e^{-(1-c_k)hM})^{\intercal}, \quad k,l=1,\ldots,s,
\end{array}
\right.\end{equation}
where  $G^{-1}_k=\nabla^2U(y_n+h\sum\limits_{l=1}^s a_{kl}\xi_l)$ and  $\xi_k =f(y_n+h\sum\limits_{l=1}^s a_{kl}\xi_l)$, then the SVERK method is symplectic.
\end{mytheo}
\begin{proof}  A numerical method is
said to be symplectic if the numerical solution $y_{n+1}$ satisfies $(\frac{\partial y_{n+1}}{\partial y_0})^{\intercal}J(\frac{\partial y_{n+1}}{\partial y_0})$ $=(\frac{\partial y_{n}}{\partial y_0})^{\intercal}J(\frac{\partial y_{n}}{\partial y_0})$.  Under the assumption $w_s(z)=0$, we rewrite the SVERK method as
\begin{equation}\label{SVERK-form2}
\left\{
\begin{array}{l}
\displaystyle \xi_{k}=f\big(e^{-c_khM}y_n+h\sum\limits_{l=1}^sa_{kl}\xi_l\big), \
i=1,\ldots,s,\cr\noalign{\vskip4truemm}
\displaystyle y_{n+1}=e^{-hM}y_n+h\sum\limits_{k=1}^sb_{k}\xi_{k}.
\end{array}
\right.
\end{equation}
Letting $\Xi_k=\frac{\partial\xi_k}{\partial y_0}$, $\Psi_{n+1}=\frac{\partial y_{n+1}}{\partial y_0}$, and $G_k=\nabla^2U\big(e^{-c_khM}y_n+h\sum\limits_{l=1}^sa_{kl}\xi_l\big)$ for $ k=1, \ldots, s$. Moreover, assuming the symmetric matrices $G_1, \ldots, G_s$ are nonsingular, we apply (\ref{SVERK-form2})
 to a Hamiltonian system (\ref{Hamiltonian}), the derivative of $y_{n+1}$ with respect to $y_0$ is
 \begin{equation}
 \Psi_{n+1}=e^{-hM}\Psi_n+h\sum\limits_{k=1}^s b_k\Xi_k.
 \end{equation}
  It is easy to see that
  \begin{equation}\label{updates symplectic of SVERK}
  \begin{aligned}
 \Psi_{n+1}^{\intercal}J \Psi_{n+1}&=\Big(e^{-hM}\Psi_n+h\sum\limits_{k=1}^s b_k\Xi_k\Big)^{\intercal}J\Big(e^{-hM}\Psi_n+h\sum\limits_{k=1}^s b_k\Xi_k\Big)\cr\noalign{\vskip4truemm}
 &=(e^{-hM}\Psi_n)^{\intercal}J(e^{-hM}\Psi_n)+h\sum\limits_{k=1}^sb_k(e^{-hM}\Psi_n)^{\intercal}J\Xi_k+h\sum\limits_{k=1}^s
 b_k\Xi_{k}^{\intercal}Je^{-hM}\Psi_n\cr\noalign{\vskip4truemm}
 &\quad +h^2\sum\limits_{k=1}^s\sum\limits_{l=1}^sb_kb_l\Xi_{k}^{\intercal}J\Xi_l.
 \end{aligned}
 \end{equation}
 It follows from (\ref{SVERK-form2}) that
 \begin{equation}
 \Xi_k=J^{-1}G_k\big(e^{-c_khM}\Psi_n+h\sum\limits_{l=1}^s a_{kl}\Xi_l\big), \quad k=1, \ldots, s,
 \end{equation}
 thus
  \begin{equation}\label{interal representation}
 \Psi_n=e^{c_khM}G_{k}^{-1}J\Xi_k-h\sum\limits_{l=1}^s a_{kl}e^{c_khM}\Xi_l, \quad k=1, \ldots, s.
 \end{equation}
In view of  (\ref{interal representation}), we have
\begin{equation}\label{formula1}
\begin{aligned}
h\sum\limits_{k=1}^sb_k(e^{-hM}\Psi_n)^{\intercal}J\Xi_k&=h\sum\limits_{l=1}^s b_l\big(e^{c_lhM}G_l^{-1}J\Xi_l-h\sum\limits_{k=1}^s a_{lk}e^{c_lhM}\Xi_k\big)^{\intercal}(e^{-hM})^{\intercal}J\Xi_l\cr\noalign{\vskip4truemm}
&=h\sum\limits_{k=1}^s b_k\Xi_{k}^{\intercal} J^{\intercal} G_{k}^{-1} (e^{-(1-c_k)hM})^{\intercal}J\Xi_{k}-
h^2\sum\limits_{k=1}^s \sum\limits_{l=1}^s b_la_{lk}\Xi_{k}^{\intercal}(e^{-(1-c_k)hM})^{\intercal}J\Xi_{l},
\end{aligned}
\end{equation}
and
\begin{equation}\label{formula2}
\begin{aligned}
h\sum\limits_{k=1}^sb_k\Xi_{k}^{\intercal} Je^{-hM}\Psi_n&=h\sum\limits_{k=1}^s b_k\Xi_k^{\intercal}Je^{-hM}\big(e^{c_khM}G_{k}^{-1}J\Xi_k-h\sum\limits_{l=1}^s a_{kl}e^{c_khM}\Xi_l\big)\cr\noalign{\vskip4truemm}
&=h\sum\limits_{k=1}^s b_k\Xi_{k}^{\intercal} J e^{-(1-c_k)hM} G_{k}^{-1}J\Xi_{k}-
h^2\sum\limits_{k=1}^s \sum\limits_{l=1}^s b_ka_{kl}\Xi_{k}^{\intercal}Je^{-(1-c_k)hM}\Xi_{l}.
\end{aligned}
\end{equation}
 Inserting  (\ref{formula1}) and (\ref{formula2}) into (\ref{updates symplectic of SVERK}) yields
 \begin{equation*}
 \begin{aligned}
 \Psi_{n+1}^{\intercal}J \Psi_{n+1}=&\Psi_{n}^{\intercal} J\Psi_{n}+h\sum\limits_{k=1}^s b_k\Xi_{k}^{\intercal}
 \Big(J^{\intercal}G_{k}^{-1}(e^{-(1-c_k)hM})^{\intercal}+J e^{-(1-c_k)hM} G_{k}^{-1}\Big)J\Xi_{k}\cr\noalign{\vskip4truemm}
&
-h^2\sum\limits_{k=1}^s \sum\limits_{l=1}^s b_la_{lk}\Xi_{k}^{\intercal}(e^{-(1-c_k)hM})^{\intercal}J\Xi_{l}-
h^2\sum\limits_{k=1}^s \sum\limits_{l=1}^s b_ka_{kl}\Xi_{k}^{\intercal}Je^{-(1-c_k)hM}\Xi_{l}
\cr\noalign{\vskip4truemm}
&
+h^2\sum\limits_{k=1}^s\sum\limits_{l=1}^sb_kb_l\Xi_{k}^{\intercal}J\Xi_l
\cr\noalign{\vskip4truemm}
=&\Psi_{n}^{\intercal} J\Psi_{n}+h\sum\limits_{k=1}^s b_k\Xi_{k}^{\intercal}J
 \Big(  e^{-(1-c_k)hM} G_{k}^{-1}-G_{k}^{-1}(e^{-(1-c_k)hM})^{\intercal}\Big)J\Xi_{k}\\
&+h^2\sum\limits_{k=1}^s \sum\limits_{l=1}^s \Xi_{k}^{\intercal} \big(b_kb_l-b_ka_{kl}Je^{-(1-c_k)hM}J^{-1}-b_la_{lk}(e^{-(1-c_k)hM})^{\intercal}\big)J\Xi_l.
 \end{aligned}
 \end{equation*}
 As the coefficients of the method satisfy the conditions (\ref{symplectic conditions}), a direct calculation leads to
 $$\Psi_{n+1}J\Psi_{n+1}=\Psi_{n}J\Psi_{n}.$$
 Therefore the SVERK method is symplectic. The proof is completed.
 $\hfill \square$
\end{proof}

\begin{rem}
  It can be observed that  the symplectic conditions of the SVERK method with $w_s(z)=0$ reduce to the  symplectic conditions of classical RK methods once $M\rightarrow \mathbf{0}$.  A  choice is that  $b_1=1$, $c_1=1$ and $a_{11}=1/2$, we can obtain  the first-order symplectic SVERK method
  \begin{equation}\label{symplectic SVERK1-1}
\left\{
\begin{array}{l}
\displaystyle Y_1=e^{-hM}y_n+\frac{h}{2}f(Y_1), \cr\noalign{\vskip4truemm}
\displaystyle y_{n+1}=e^{-hM}y_n+hf(Y_1).
\end{array}
\right.
\end{equation}
For \eqref{equ1}, the method \eqref{symplectic SVERK1-1} reduces to the implicit midpoint method when $M\rightarrow \mathbf{0}$. Unfortunately, there is no existing the symplectic SVERK method  with order  $p\geq2$ due to $w_s(z)\neq0$.
\end{rem}
   The next theorem will present  the symplectic conditions of the  MVERK method with $\bar{w}_s(z)=0$.
\begin{mytheo}\label{symplectic MVERK}
Assume the coefficients of an $s$-stage MVERK method with $\bar{w}_s(z)=0$, which satisfy the following conditions:
\begin{equation}\label{symplectic conditions of MVERK}\left\{
\begin{array}{l}
D_i^{\intercal}Je^{-hM}+(e^{-hM})^{\intercal}JD_i=0, \ i=1,\ldots, s,\cr\noalign{\vskip1truemm}
\bar{b}_i\bar{b}_j=\bar{b}_i\bar{a}_{ij}Je^{-hM}J^{-1}+\bar{b}_j\bar{a}_{ji}e^{-hM}, \ i,j=1,\ldots,s,\cr\noalign{\vskip1truemm}
\bar{b}_j\bar{a}_{ji}M^{\intercal}(e^{-hM})^{\intercal}JD_j+\bar{b}_i\bar{a}_{ij}D_i^{\intercal}Je^{-hM}M=0,\ i,j=1,\ldots,s,
\end{array}
\right.\end{equation}
with $D_i=\frac{\partial f(\bar{Y}_i)}{\partial y_0}$, then the MVERK method is symplectic.
\end{mytheo}

\begin{proof}
Setting $D_i=\frac{\partial f(\bar{Y}_i)}{\partial y}$, $X_i=\frac{\partial \bar{Y}_i}{\partial y_0}$, and $\Psi_{n+1}=\frac{\partial \bar{y}_{n+1}}{y_0}$. Similarly, we apply the MVERK method    \eqref{MVERK}  with $\bar{w}_s(z)=0$ to a Hamiltonian system (\ref{Hamiltonian}),  the derivative of this scheme with respect to $y_0$ is
\begin{equation}\label{the derivative of MVERK}\left\{
\begin{array}{l}
X_i=\frac{\partial \bar{Y}_i}{\partial y_0}=\Psi_{n}+h\sum\limits_{j=1}^s \bar{a}_{ij}(-MX_j+D_jX_j),\cr\noalign{\vskip1truemm}
\Psi_{n+1}=e^{-hM}\Psi_n+h\sum\limits_{i=1}^s \bar{b}_iD_iX_i.
\end{array}
\right.\end{equation}
Then, we have
  \begin{equation}\label{updates symplectic}
  \begin{aligned}
 \Psi_{n+1}^{\intercal}J \Psi_{n+1}&=\Big(e^{-hM}\Psi_n+h\sum\limits_{i=1}^s \bar{b}_iD_iX_i\Big)^{\intercal}J\Big(e^{-hM}\Psi_n+h\sum\limits_{i=1}^s \bar{b}_iD_iX_i\Big)\\
 &=(e^{-hM}\Psi_n)^{\intercal}J(e^{-hM}\Psi_n)+h\sum\limits_{i=1}^s\bar{b}_i(e^{-hM}\Psi_n)^{\intercal}JD_iX_i+h\sum\limits_{i=1}^s
 \bar{b}_i(D_iX_i)^{\intercal}Je^{-hM}\Psi_n\\
  &\quad +h^2\sum\limits_{i=1}^s\sum\limits_{j=1}^s\bar{b}_i\bar{b}_j(D_iX_i)^{\intercal}JD_jX_j.
  \end{aligned}
 \end{equation}
 Using the first formula of (\ref{the derivative of MVERK}), we obtain
 \begin{equation}
 \begin{aligned}
X_i^{\intercal}(e^{-hM})^{\intercal} JD_iX_i&=\Psi_n^{\intercal}(e^{-hM})^{\intercal} JD_iX_i+h\sum\limits_{j=1}^s
\bar{a}_{ij}(-MX_j+D_jX_j)^{\intercal}(e^{-hM})^{\intercal}JD_iX_i,\cr\noalign{\vskip4truemm}
(D_iX_i)^{\intercal} Je^{-hM}X_i&=(D_iX_i)^{\intercal} Je^{-hM}\Psi_n+h\sum\limits_{j=1}^s \bar{a}_{ij}(D_iX_i)^{\intercal} Je^{-hM}(-MX_j+D_jX_j),
\end{aligned}
 \end{equation}
 thus
  \begin{equation}\label{MVERK formula1}
 \begin{aligned}
\Psi_n^{\intercal}(e^{-hM})^{\intercal} JD_iX_i&=X_i^{\intercal}(e^{-hM})^{\intercal} JD_iX_i-h\sum\limits_{j=1}^s
\bar{a}_{ij}(-MX_j+D_jX_j)^{\intercal}(e^{-hM})^{\intercal}JD_iX_i,\cr\noalign{\vskip4truemm}
(D_iX_i)^{\intercal} Je^{-hM}\Psi_n&=(D_iX_i)^{\intercal} Je^{-hM}X_i-h\sum\limits_{j=1}^s \bar{a}_{ij}(D_iX_i)^{\intercal} Je^{-hM}(-MX_j+D_jX_j).
\end{aligned}
 \end{equation}
Inserting (\ref{MVERK formula1}) into (\ref{updates symplectic}) leads to
 \begin{equation*}
 \begin{aligned}
 \Psi_{n+1}^{\intercal}J \Psi_{n+1}=&\Psi_{n}^{\intercal} J\Psi_{n}+h\sum\limits_{i=1}^s \bar{b}_iX_i^{\intercal}
 (e^{-hM})^{\intercal}JD_iX_i-h^2\sum\limits_{i=1}^s \sum\limits_{j=1}^s \bar{b}_j\bar{a}_{ji}(-MX_i+D_iX_i)^{\intercal}(e^{-hM})^{\intercal}JD_j \cr\noalign{\vskip4truemm}
&\cdot X_j
+h\sum\limits_{i=1}^s \bar{b}_i(D_iX_i)^{\intercal}Je^{hM}X_i-h^2\sum\limits_{i=1}^s \sum\limits_{j=1}^s \bar{b}_i\bar{a}_{ij}(D_iX_i)^{\intercal}JS(-MX_j+D_jX_j)\cr\noalign{\vskip4truemm}
\end{aligned}
\end{equation*}
 \begin{equation*}
 \begin{aligned}
&
+h^2\sum\limits_{i=1}^s \sum\limits_{j=1}^s\bar{b}_i\bar{b}_j(D_iX_i)^{\intercal}J(D_jX_j)
\cr\noalign{\vskip4truemm}
=&\Psi_{n}^{\intercal} J\Psi_{n}+h\sum\limits_{i=1}^s \bar{b}_iX_i^{\intercal}\big(D_i^{\intercal}Je^{hM}+(e^{hM})^{\intercal}JD_i\big)X_i\cr\noalign{\vskip4truemm}
 &+h^2\sum\limits_{i=1}^s \sum\limits_{j=1}^s (D_iX_i)^{\intercal}\big(\bar{b}_i\bar{b}_j-\bar{b}_j\bar{a}_{ij}(e^{-hM})^{\intercal}
-\bar{b}_i\bar{a}_{ij}JSJ^{-1}\big)JD_jX_j\cr\noalign{\vskip4truemm}
&+h^2\sum\limits_{i=1}^s \sum\limits_{j=1}^s X_i^{\intercal}\big(\bar{b}_j\bar{a}_{ji}M^{\intercal} (e^{-hM})^{\intercal}
JD_j+\bar{b}_i\bar{a}_{ij}D_i^{\intercal} Je^{-hM}M\big)X_j.
 \end{aligned}
\end{equation*}
As the coefficients of the MVERK method satisfy conditions (\ref{symplectic conditions of MVERK}), it can be verified  that
$$\Psi_{n+1}^{\intercal}J\Psi_{n+1}=\Psi_{n}^{\intercal}J\Psi_{n},$$
whence the MVERK method is symplectic. The proof is completed.
$\hfill \square$
\end{proof}

\begin{rem} Theorem \ref{symplectic MVERK} presents the symplectic conditions of the MVERK method with $w_s(z)=0$. However, when $b_i$ ($i=1,\ldots,s$) are constants,  the second formula of (\ref{symplectic conditions of MVERK}) can never be satisfied. The symplectic conditions of
the MVERK method with $w_s(z)\neq0$ can be analyzed by the same way, unfortunately, the MVERK methods with $\bar{w}_s(z)\neq0$ can not  preserve the symplectic invariant.
\end{rem}
\begin{cor} There does not exist any symplectic MVERK methods.
\end{cor}

\section{Highly accurate implicit ERK methods}\label{sec3}

Section \ref{sec2}  is concerned with the symplectic conditions for the SVERK and MVERK methods. However,  the symplectic SVERK method  only has order one, and there is no existing  symplectic MVERK methods. In practice, we need some highly accurate and effective numerical methods, hence the high-order implicit SVERK (IMSVERK) and implicit MVERK (IMMVERK) methods are studied in this section.

We now consider a one-stage IMSVERK method with $w_2(z)=-\frac{h^2Mf(y_0)}{2!}$
 \begin{equation}\label{SVERK12}
\left\{
\begin{array}{l}
\displaystyle Y_1=e^{-c_1hM}y_0+ha_{11}f(Y_1), \cr\noalign{\vskip4truemm}
\displaystyle y_{1}=e^{-hM}y_0+hb_{1}f(Y_1)-\frac{h^2Mf(y_0)}{2!},
\end{array}
\right.
\end{equation}
and  a one-stage IMMVERK method with $\bar{w}_2(z)=-\frac{h^2Mf(y_0)}{2!}$
 \begin{equation}\label{MVERK21}
\left\{
\begin{array}{l}
\displaystyle \bar{Y}_1=y_0+h\bar{a}_{11}(-M\bar{Y}_1+f(\bar{Y}_1)), \cr\noalign{\vskip4truemm}
\displaystyle \bar{y}_{1}=e^{-hM}y_0+h\bar{b}_1f(\bar{Y}_1)-\frac{h^2Mf(y_0)}{2!}.
\end{array}
\right.
\end{equation}
 A numerical method is said to be of order $p$
if the Taylor expansions of numerical solutions $y_1$ or $\bar{y}_1$ and exact solution $y(t_0+h)$ coincides up to $h^p$ about $y_0$.
For convenience, we denote $g(t_0)=-My(t_0)+f(y(t_0))$, the Taylor expansion for exact solution $y(t_0+h)$ is
\begin{align*}
\displaystyle y&(t_0+h)=y(t_0)+hy'(t_0)+\frac{h^2}{2!}y''(t_0)+\frac{h^3}{3!}y'''(t_0)+\frac{h^4}{4!}y^{(4)}(t_0)
+\mathcal{O}{(h^5)}\cr\noalign{\vskip4truemm} &=y(t_0)+hg(t_0)+\frac{h^2}{2!}(-M+f'_y(y_0))g(t_0)+\frac{h^3}{3!}\big(M^2g(t_0)+(-M+f'_y(y(t_0)))f'_y(y(t_0))g(t_0)
\cr\noalign{\vskip4truemm}
&\quad-f'_y(y(t_0))Mg(t_0)+f''_{yy}(y(t_0))(g(t_0),g(t_0))\big)+\frac{h^4}{4!}\big(-M^3g(t_0)+M^2f'_y(y(t_0))g(t_0)-Mf'_y(y(t_0))\cr\noalign{\vskip4truemm}
&\quad\cdot(-M+f'_y(y(t_0)))g(t_0)-Mf''_{yy}(y(t_0))(g(t_0),g(t_0))+f'''_{yyy}(y(t_0))(g(t_0),g(t_0),g(t_0))+3f''_{yy}(y(t_0))\cr\noalign{\vskip4truemm}
&\quad \cdot ((-M+f'_y(y(t_0)))g(t_0),g(t_0)) +f'_y(y(t_0))(-M+f'_y(y(t_0)))(-M+f'_y(y(t_0)))g(t_0)\cr\noalign{\vskip4truemm}
&\quad+f'_y(y(t_0))f''_{yy}(y(t_0))(g(t_0),g(t_0))\big)+\mathcal{O}{(h^5)}.
\end{align*}
Under the assumption $y_0=y(t_0)$, the Taylor expansions for numerical solutions $y_1$ and $\bar{y}_1$ are
\begin{equation*}
\begin{aligned}
y_1&=(I-hM+\frac{h^2M^2}{2!}+\mathcal{O}{(h^3)})y_0+hb_1f(y_0-c_1hMy_0+a_{11}hf(y_0)
+\mathcal{O}(h^2))-\frac{h^2Mf(y_0)}{2!}\cr\noalign{\vskip4truemm}
&=y_0-hMy_0+\frac{h^2M^2}{2!}y_0+hb_1f(y_0)+h^2b_1c_1f'_y(y_0)(-My_0+f(y_0))-\frac{h^2Mf(y_0)}{2!}+\mathcal{O}{(h^3)}\cr\noalign{\vskip4truemm}
&=y_0-hMy_0+hb_1f(y_0)+\frac{h^2}{2!}(-M)(-My_0+f(y_0))+h^2b_1c_1f'_y(y_0)(-My_0+f(y_0))+\mathcal{O}{(h^3)},
\end{aligned}
\end{equation*}
and
\begin{equation*}
\begin{aligned}
\bar{y}_1&=(I-hM+\frac{h^2M^2}{2!}+\mathcal{O}{(h^3)})y_0+h\bar{b}_1\big[f(y_0)+h\bar{a}_{11}f'_y(y_0)(-My_0
+f(y_0)+\mathcal{O}{(h)})\big]-\frac{h^2Mf(y_0)}{2!}\cr\noalign{\vskip4truemm}
&=y_0-hMy_0+h\bar{b}_1f(y_0)+\frac{h^2M^2}{2!}y_0+h^2\bar{b}_1\bar{a}_{11}f'_y(y_0)(-My_0+f(y_0))-\frac{h^2Mf(y_0)}{2!}+\mathcal{O}{(h^3)}\cr\noalign{\vskip4truemm}
&=y_0-hMy_0+h\bar{b}_1f(y_0)+\frac{h^2}{2!}(-M)(-My_0+f(y_0))+h^2\bar{b}_1\bar{a}_{11}f'_y(y_0)(-My_0+f(y_0))+\mathcal{O}(h^3).
\end{aligned}
\end{equation*}

If we consider the  implicit second-order ERK methods
 with one stage,  then  $b_1=\bar{b}_1=1$ and $a_{11}=\bar{a}_{11}=\frac{1}{2}$. Therefore, the  second-order IMSVERK method with one stage is given by
 \begin{equation}\label{SVERK21}
\left\{
\begin{array}{l}
\displaystyle Y_1=e^{-\frac{1}{2}hM}y_0+\frac{h}{2}f(Y_1),\cr\noalign{\vskip4truemm}
\displaystyle y_{1}=e^{-hM}y_0+hf(Y_1)-\frac{h^2Mf(y_0)}{2!},
\end{array}
\right.
\end{equation}
which can be denoted by the Butcher tableau
\begin{equation}\label{Tableau-SVERK12}
\begin{aligned}
\begin{tabular}{c|c|c}
 $\frac{1}{2}$&$e^{-\frac{1}{2}hM}$&$\frac{1}{2}$ \\[3pt]
 \hline
$\raisebox{-1.3ex}[1.0pt]{$e^{-hM}$}$& $\raisebox{-1.3ex}[1.0pt]{$w_2(z)$}$ & $\raisebox{-1.3ex}[1.0pt]{$1$}$\\
\end{tabular}.
\end{aligned}
\end{equation}
The  second-order IMMVERK method with one stage is shown as
 \begin{equation}\label{MVERK22}
\left\{
\begin{array}{l}
\displaystyle \bar{Y}_1=y_0+\frac{h}{2}(-M\bar{Y}_1+f(\bar{Y}_1)), \cr\noalign{\vskip4truemm}
\displaystyle y_{1}=e^{-hM}y_0+hf(\bar{Y}_1)-\frac{h^2Mf(y_0)}{2!},
\end{array}
\right.
\end{equation}
which also can be indicated by the Butcher tableau
\begin{equation}\label{Tableau-MVERK12}
\begin{aligned}
\begin{tabular}{c|c|c}
 $\frac{1}{2}$&$I$&$\frac{1}{2}$ \\[3pt]
 \hline
$\raisebox{-1.3ex}[1.0pt]{$e^{-hM}$}$& $\raisebox{-1.3ex}[1.0pt]{$\bar{w}_2(z)$}$ & $\raisebox{-1.3ex}[1.0pt]{$1$}$\\
\end{tabular}.
\end{aligned}
\end{equation}

 ERK methods of collocation type were formulated by Hochbruck  and Ostermann, and their convergence properties also were analyzed in  \cite{Hochbruck2005b}.  Using the collocation code $c_1=1/2$, then the implicit second-order (stiff) ERK method with one stage is
 \begin{equation}\label{Tableau-ERK12}
\begin{aligned}
\begin{tabular}{c|c}
 $\frac{1}{2}$&$\frac{1}{2}\varphi_1(-\frac{1}{2}hM)$\\[3pt]
 \hline
& $\raisebox{-1.3ex}[1.0pt]{$\varphi_1(-hM)$}$\\
\end{tabular}
\end{aligned}
\end{equation}
with
\begin{equation}\label{entire functions}
\varphi_{ij}(-hM)=\varphi_i(-c_jhM)=\int_0^1 e^{-(1-\tau)c_jhM}\frac{\tau^{i-1}}{(i-1)!}d\tau.
\end{equation}
We notice that  when $M\rightarrow \mathbf{0}$, the methods  \eqref{Tableau-SVERK12}, \eqref{Tableau-MVERK12},
and \eqref{Tableau-ERK12} reduce to the implicit midpoint rule.

 Now, we consider the fourth-order IMSVERK and IMMVERK methods with two stages.  The order conditions of the implicit fourth-order  RK method with two stages are
\begin{equation}\label{order conditions}
\begin{aligned}
 \includegraphics{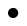} \qquad &b_1+b_2&=1, \cr\noalign{\vskip4truemm}
\includegraphics{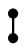} \qquad & b_1c_1+b_2c_2&=\frac{1}{2},\cr\noalign{\vskip4truemm}
\includegraphics{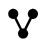} \qquad  &b_1c_1^2+b_2c_2^2&=\frac{1}{3},\cr\noalign{\vskip4truemm}
\includegraphics{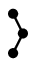} \qquad &b_1(a_{11}c_1+a_{12}c_2)+b_2(a_{21}c_1+a_{22}c_2)&=\frac{1}{6},\cr\noalign{\vskip4truemm}
 \includegraphics{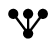} \qquad&b_1c_1^3+b_2c_2^3&=\frac{1}{4},\cr\noalign{\vskip4truemm}
 \includegraphics{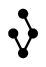} \qquad&b_1c_1(a_{11}c_1+a_{12}c_2)+b_2c_2(a_{21}c_1+a_{22}c_2)&=\frac{1}{8},\cr\noalign{\vskip4truemm}
 \includegraphics{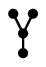} \qquad&b_1(a_{11}c_1^2+a_{12}c_2^2)+b_2(a_{21}c_1^2+a_{22}c_2^2)&=\frac{1}{12},\cr\noalign{\vskip4truemm}
 \includegraphics{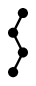} \qquad&(b_1a_{11}+b_2a_{21})(a_{11}c_1+a_{12}c_2)+(b_1a_{12}+b_2a_{22})(a_{21}c_1+a_{22}c_2)&=\frac{1}{24}.\cr\noalign{\vskip4truemm}
 \end{aligned}
\end{equation}
Under the assumptions $c_1=a_{11}+a_{12}$ and $c_2=a_{21}+a_{22}$, it has  a unique solution \cite{ Hammer1955,Hairer2006}.  The following theorem verifies that the  order conditions of fourth-order IMSVERK methods with two stages are identical to (\ref{order conditions}).

\begin{mytheo}\label{the proof of SVERK24}
Assume  the coefficients of a two-stage IMSVERK method with $w_4(z)$
 \begin{equation}\label{SVERK2-4}
\left\{
\begin{array}{l}
\displaystyle
Y_1=e^{-c_1hM}y_0+h\big[a_{11}f(Y_1)+a_{12}f(Y_2)\big],\cr\noalign{\vskip4truemm}
Y_2=e^{-c_2hM}y_0+h\big[a_{21}f(Y_1)+a_{22}f(Y_2)\big],\cr\noalign{\vskip4truemm}
\displaystyle y_{1}=e^{-hM}y_0+h(b_1f(Y_1)+b_2f(Y_2))+w_4(z),
\end{array}
\right.
\end{equation}
where
\begin{equation*}
\begin{aligned}
w_4(z)&=-\frac{h^2}{2!}Mf(y_0)+\frac{h^3}{3!}\big((M-f'_y(y_0))Mf(y_0)
-Mf'_y(y_0)g(y_0)\big) +\frac{h^4}{4!}\Big((-M+f'_y(y_0))M^2f(y_0)\\ &\quad+M^2f'_y(y_0)g(y_0)-Mf''_{yy}(y_0)\big(g(y_0),g(y_0)\big)-Mf'_y(y_0)(-M+f'_y(y_0))g(y_0)-f'_y(y_0)Mf'_y(y_0)\\
&\quad \cdot g(y_0)-f'_y(y_0)f'_y(y_0)Mf(y_0)+3f''_{yy}(y_0)\big(-Mf(y_0),g(y_0)\big)\Big),
\end{aligned}
\end{equation*}
which satisfy (\ref{order conditions}), then  the IMSVERK method has order four.
\end{mytheo}
\begin{proof}  Ignoring the term $\mathcal{O}(h^5)$,  and the Taylor  expansion of numerical solution $y_1$ is shown by
\begin{equation*}
y_1=e^{-hM}y_0+hb_1f\big((I-c_1hM+\frac{(c_1hM)^2}{2!}-
\frac{(c_1hM)^3}{3!}+\mathcal{O}(h^4))y_0+ha_{11}f(Y_1)+ha_{12}f(Y_2)\big)+hb_2
\end{equation*}
\begin{equation}
\begin{aligned}
&\cdot f\big((I-c_2hM+\frac{(c_2hM)^2}{2!}-\frac{(c_2hM)^3}{3!}+\mathcal{O}(h^4))y_0+ha_{21}f(Y_1)+ha_{22}f(Y_2)
\big)+w_4(z)+\mathcal{O}(h^5)\cr\noalign{\vskip4truemm}
=&(I-hM+\frac{h^2M^2}{2}-\frac{h^3M^3}{3!}+\frac{h^4M^4}{4!}+\mathcal{O}(h^5))y_0+h(b_1+b_2)f(y_0)
+h^2b_1f^{\prime}_y(y_0)(-c_1My_0 +a_{11} \cr\noalign{\vskip4truemm}
&\cdot f(Y_1)+a_{12}f(Y_2)+\frac{h(c_1M)^2}{2}y_0-\frac{h^2(c_1M)^3}{3!}y_0
+\mathcal{O}(h^3))+h^2b_2f^{\prime}_y(y_0)(-c_2My_0+a_{21}f(Y_1)\cr\noalign{\vskip4truemm}
& +a_{22}f(Y_2)+\frac{h(c_2M)^2}{2}y_0-\frac{h^2(c_2M)^3}{3!}y_0+\mathcal{O}(h^3))+\frac{h^3}{2}b_1f''_{yy}(y_0)\big(-c_1My_0+a_{11}f(Y_1)
+a_{12}\cr\noalign{\vskip4truemm}
&\cdot f(Y_2) +\frac{h(c_1M)^2}{2}y_0,-c_1My_0+a_{11}f(Y_1)+a_{12}f(Y_2)+\frac{h(c_1M)^2}{2}y_0\big)+\frac{h^3}{2}b_2f''_{yy}(y_0)\big(-c_2My_0
\cr\noalign{\vskip4truemm}
&+a_{21}f(Y_1)+a_{22}f(Y_2)+\frac{h(c_2M)^2}{2}y_0,-c_2My_0+a_{11}f(Y_1)+a_{12}f(Y_2)+\frac{h(c_2M)^2}{2}y_0\big)+\frac{h^4}{3!}b_1\cr\noalign{\vskip4truemm}
&\cdot f'''_{yyy}(y_0)\big(-c_1M y_0+a_{11}f(Y_1)+a_{12}f(Y_2),-c_1My_0+a_{11}f(Y_1)+a_{12}f(Y_2),-c_1My_0+a_{11}\cr\noalign{\vskip4truemm}
&\cdot f(Y_1)+a_{12}f(Y_2)\big)+\frac{h^4}{3!}b_2f'''_{yyy}(y_0)(-c_2My_0+a_{21}f(Y_1)+a_{22}f(Y_2),-c_2My_0+a_{21}f(Y_1)\cr\noalign{\vskip4truemm}
&+a_{22}f(Y_2),-c_2My_0+a_{21}f(Y_1)+a_{22}f(Y_2))+w_4(z)+\mathcal{O}(h^5).
\end{aligned}
\end{equation}
Under the  appropriate assumptions $c_i=\sum\limits_{j=1}^s a_{ij}$ for $i,j=1,\ldots,s$, we have
\begin{equation*}\label{Taylor series of N}
\begin{aligned}
y_1=&e^{-hM}y_0+h(b_1+b_2)f(y_0)+h^2b_1f^{\prime}_y(y_0)\big(-c_1My_0+a_{11}f(Y_1)+a_{12}f(Y_2)+\frac{h(c_1M)^2}{2}y_0
\cr\noalign{\vskip4truemm}
&-\frac{h^2(c_1M)^3}{3!}y_0\big)+h^2b_2f^{\prime}_y(y_0)\big(-c_2My_0+a_{21}f(Y_1)+a_{22}f(Y_2)+\frac{h(c_2M)^2}{2}y_0
-\frac{h^2(c_2M)^3}{3!}y_0\big)\cr\noalign{\vskip4truemm}
&+\frac{h^3}{2}b_1f''_{yy}(y_0)\big(-c_1My_0+a_{11}f(Y_1)+a_{12}f(Y_2)+\frac{h(c_1M)^2}{2}y_0,
-c_1My_0+a_{11}f(Y_1)+a_{12}f(Y_2)\cr\noalign{\vskip4truemm}
&+\frac{h(c_1M)^2}{2}y_0\big)+\frac{h^3}{2}b_2f''_{yy}(y_0)\big(-c_2My_0+a_{21}f(Y_1)+a_{22}f(Y_2)+\frac{h(c_2M)^2}{2}y_0,
-c_2My_0+a_{21}\cr\noalign{\vskip4truemm}
\end{aligned}
\end{equation*}
\begin{equation*}
\begin{aligned}
&\cdot f(Y_1)+a_{22}f(Y_2)+\frac{h(c_2M)^2}{2}y_0\big)+\frac{h^4}{3!}b_1f'''_{yyy}(y_0)\big(c_1(-My_0+f(y_0)),c_1(-My_0+f(y_0)),c_1\cr\noalign{\vskip4truemm}
&\cdot (-My_0+f(y_0))\big) +\frac{h^4}{3!}b_2f'''_{yyy}(y_0)\big(c_2(-My_0+f(y_0)),c_2(-My_0+f(y_0)),c_2(-My_0+f(y_0))\big)\cr\noalign{\vskip4truemm}
&+w_4(z)+\mathcal{O}(h^5).
\end{aligned}
\end{equation*}
Inserting $Y_1$ and $Y_2$ of (\ref{SVERK2-4}) into $f(Y_1)$ and $f(Y_2)$ yields
\begin{equation*}
\begin{aligned}
a_{11}f(Y_1)&=a_{11}f\big(y_0-c_1hMy_0+ha_{11}f(Y_1)+ha_{12}f(Y_2)+\frac{(c_1hM)^2}{2!}y_0\big)\cr\noalign{\vskip4truemm}
&=a_{11}f(y_0)+hf^{\prime}_y(y_0)\big(-c_1My_0+a_{11}f(y_0)+ha_{11}f^{\prime}_y(y_0)(-c_1My_0+a_{11}f(y_0)+a_{12}f(y_0))\cr\noalign{\vskip4truemm}
&\quad +a_{12}f(y_0)+ha_{12}f^{\prime}_y(y_0)(-c_2My_0+a_{21}f(y_0)+a_{22}f(y_0))+\frac{h(c_1M)^2}{2}y_0\big)+\frac{h^2}{2}a_{11}c_1^2 \cr\noalign{\vskip4truemm}
&\quad \cdot f''_{yy}(y_0)(g(y_0),g(y_0))\cr\noalign{\vskip4truemm}
&=a_{11}f(y_0)+ha_{11}c_1f^{\prime}_{y}(y_0)g(y_0)+h^2(a_{11}^2c_1+a_{11}a_{12}c_2)f^{\prime}_{y}(y_0)f^{\prime}_{y}(y_0)g(y_0)
+\frac{h^2}{2}a_{11}c_1^2f^{\prime}_{y}(y_0)\cr\noalign{\vskip4truemm}
&\quad \cdot M^2y_0+\frac{h^2}{2}a_{11}c_1^2f''_{yy}(y_0)(g(y_0),g(y_0)),
\end{aligned}
\end{equation*}
and
\begin{equation*}
\begin{aligned}
a_{12}f(Y_2)&=a_{12}f(y_0)+ha_{12}c_2f^{\prime}_{y}(y_0)g(y_0)+h^2(a_{12}a_{21}c_1+a_{12}a_{22}c_2)f^{\prime}_{y}(y_0)f^{\prime}_{y}(y_0)g(y_0)
+\frac{h^2}{2}a_{12}c_2^2\cr\noalign{\vskip4truemm}
&\quad \cdot f^{\prime}_{y}(y_0)M^2y_0+\frac{h^2}{2}a_{12}c_2^2f''_{yy}(y_0)(g(y_0),g(y_0)).
\end{aligned}
\end{equation*}
Hence, we get
\begin{equation*}
\begin{aligned}
&h^2b_1f^{\prime}(y_0)\big(-c_1My_0+a_{11}f(Y_1)+a_{12}f(Y_2)+\frac{h(c_1M)^2}{2}y_0
-\frac{h^2(c_1M)^3}{3!}y_0\big)=h^2b_1c_1f^{\prime}_y(y_0)g(y_0)\cr\noalign{\vskip4truemm}
&+\frac{h^3}{2}b_1c_1^2f^{\prime}_y(y_0) M^2y_0-\frac{h^4}{3!}b_1c_1^3f^{\prime}_y(y_0)M^3y_0+h^3b_1(a_{11}c_1+a_{12}c_2)f^{\prime}_y(y_0)f^{\prime}_y(y_0)g(y_0)+h^4b_1(a_{11}^2c_1
\cr\noalign{\vskip4truemm}
&+a_{11}a_{12}c_2+a_{12}a_{21}c_1+a_{12}a_{22}c_2)f^{\prime}_y(y_0)f^{\prime}_y(y_0)f^{\prime}_y(y_0)g(y_0)
+\frac{h^4}{2}b_1(a_{11}c_1^2+a_{12}c_2^2)\big(f^{\prime}_y(y_0)f^{\prime}_y(y_0)M^2\cr\noalign{\vskip4truemm}
&\cdot y_0+f^{\prime}_y(y_0)f''_{yy}(y_0)(g(y_0),g(y_0))\big),
\end{aligned}
\end{equation*}
and
\begin{equation*}
\begin{aligned}
&h^2b_2f^{\prime}(y_0)\big(-c_2My_0+a_{21}f(Y_1)+a_{22}f(Y_2)+\frac{h(c_2M)^2}{2}y_0
-\frac{h^2(c_2M)^3}{3!}y_0\big)=h^2b_2c_2f^{\prime}_y(y_0)g(y_0)\cr\noalign{\vskip4truemm}
&+\frac{h^3}{2}b_2c_2^2f^{\prime}_y(y_0) M^2y_0-\frac{h^4}{3!}b_2c_2^3f^{\prime}_y(y_0)M^3y_0+h^3b_2(a_{21}c_1+a_{22}c_2)f^{\prime}_y(y_0)f^{\prime}_y(y_0)g(y_0)+h^4b_1(a_{22}^2c_2
\cr\noalign{\vskip4truemm}
\end{aligned}
\end{equation*}
\begin{equation*}
\begin{aligned}
&+a_{22}a_{21}c_1+a_{21}a_{12}c_2+a_{21}a_{11}c_1)f^{\prime}_y(y_0)f^{\prime}_y(y_0)f^{\prime}_y(y_0)g(y_0)
+\frac{h^4}{2}b_2(a_{21}c_1^2+a_{22}c_2^2)\big(f^{\prime}_y(y_0)f^{\prime}_y(y_0)M^2\cr\noalign{\vskip4truemm}
&\cdot y_0+f^{\prime}_y(y_0)f''_{yy}(y_0)(g(y_0),g(y_0))\big).
\end{aligned}
\end{equation*}
Likewise, a direct calculation leads to
\begin{equation*}
\begin{aligned}
&\frac{h^3}{2}b_1f''_{yy}(y_0)\big(-c_1My_0+a_{11}f(Y_1)+a_{12}f(Y_2)+\frac{h(c_1M)^2}{2}y_0,-c_1My_0+a_{11}f(Y_1)+a_{12}f(Y_2)\cr\noalign{\vskip4truemm}
&\quad +\frac{h(c_1M)^2}{2}y_0\big)\cr\noalign{\vskip4truemm}
&=\frac{h^3}{2}b_1f''_{yy}(y_0)\big(c_1g(y_0)+h(a_{11}c_1+a_{12}c_2)f^{\prime}_y(y_0)g(y_0)+\frac{h(c_1M)^2}{2}y_0,c_1g(y_0)+h(a_{11}c_1+a_{12}c_2)\cr\noalign{\vskip4truemm}
&\quad \cdot f^{\prime}_y(y_0)g(y_0)+\frac{h(c_1M)^2}{2}y_0)\big)\cr\noalign{\vskip4truemm}
&=\frac{h^3}{2}b_1c_1^2f''_{yy}(y_0)(g(y_0),g(y_0))+h^4b_1(a_{11}c_1+a_{12}c_2)c_1f''_{yy}(y_0)(f^{\prime}_y(y_0)g(y_0),g(y_0))\cr\noalign{\vskip4truemm}
&\quad +\frac{h^4}{2}b_1c_1^3f''_{yy}(y_0)(M^2y_0,g(y_0))
\end{aligned}
\end{equation*}
and
\begin{equation*}
\begin{aligned}
&\frac{h^3}{2}b_2f''_{yy}(y_0)\big(-c_2My_0+a_{21}f(Y_1)+a_{22}f(Y_2)+\frac{h(c_2M)^2}{2}y_0,-c_2My_0+a_{21}f(Y_1)+a_{22}f(Y_2)\cr\noalign{\vskip4truemm}
&\quad +\frac{h(c_2M)^2}{2}y_0\big)\cr\noalign{\vskip4truemm}
&=\frac{h^3}{2}b_2c_2^2f''_{yy}(y_0)(g(y_0),g(y_0))+h^4b_2(a_{21}c_1+a_{22}c_2)c_2f''_{yy}(y_0)(f^{\prime}_y(y_0)g(y_0),g(y_0))\cr\noalign{\vskip4truemm}
&\quad +\frac{h^4}{2}b_2c_2^3f''_{yy}(y_0)(M^2y_0,g(y_0)).
\end{aligned}
\end{equation*}
Summing up, the Taylor expansion for numerical solution $y_1$ is
\begin{equation*}
\begin{aligned}
y_1&=y_0-hMy_0+h(b_1+b_2)f(y_0)-\frac{h^2Mg(y_0)}{2!}+h^2(b_1c_1+b_2c_2)f^{\prime}_y(y_0)g(y_0)-\frac{h^3}{3!}f^{\prime}_y(y_0)Mf(y_0)
\cr\noalign{\vskip4truemm}
&\quad+\frac{h^3}{3!}M(M-f^{\prime}_y(y_0))g(y_0)+\frac{h^3}{2}(b_1c_1^2+b_2c_2^2)\big(f^{\prime}_y(y_0)M^2y_0+f''_{yy}(y_0)(g(y_0),g(y_0))\big)\cr\noalign{\vskip4truemm}
&\quad +h^3(b_1(a_{11}c_1+a_{12}c_2)+b_2(a_{21}c_1+a_{22}c_2))
f^{\prime}_y(y_0)f^{\prime}_y(y_0)g(y_0)\cr\noalign{\vskip4truemm}
&\quad +\frac{h^4}{4!}\big(-M^3g(y_0)+M^2f^{\prime}_y(y_0)g(y_0)
-Mf^{\prime}_y(y_0)(-M+f^{\prime}_y(y_0))g(y_0)-Mf''_{yy}(y_0)(g(y_0),g(y_0))\cr\noalign{\vskip4truemm}
&\quad-f^{\prime}_y(y_0)Mf^{\prime}_y(y_0)g(y_0)-f^{\prime}_y(y_0)f^{\prime}_y(y_0)Mf(y_0)+f^{\prime}_y(y_0)M^2f(y_0)+3f''_{yy}(y_0)(-Mf(y_0),g(y_0))\big)\cr\noalign{\vskip4truemm}
\end{aligned}
\end{equation*}
\begin{equation*}
\begin{aligned}
&\quad +h^4(b_1c_1^3+b_2c_2^3)\big(\frac{1}{3!}f'''_{yyy}(y_0)(g(y_0),g(y_0),g(y_0))+\frac{1}{2}f''_{yy}(y_0)(M^2y_0,g(y_0))-\frac{1}{3!}f^{\prime}_y(y_0)M^3y_0\big)\cr\noalign{\vskip4truemm}
&\quad+h^4(b_1c_1(a_{11}c_1+a_{12}c_2)+b_2c_2(a_{21}c_1+a_{22}c_2))f''_{yy}(y_0)(f^{\prime}_y(y_0)g(y_0),g(y_0))\cr\noalign{\vskip4truemm}
&\quad +\frac{h^4}{2}\big(b_1(a_{11}c_1^2+a_{12}c_2^2)+b_2(a_{21}c_1^2+a_{22}c_2^2)\big)f^{\prime}_y(y_0)(f''_{yy}(y_0)(g(y_0),g(y_0))+f^{\prime}_y(y_0)f^{\prime}_y(y_0)M^2y_0)\cr\noalign{\vskip4truemm}
&\quad+h^4\big(b_1(a_{11}^2c_1+(a_{11}+a_{22})a_{12}c_2+a_{12}a_{21}c_1)+b_2(a_{22}^2c_2+(a_{11}+a_{22})a_{21}c_1+a_{21}a_{12}c_2)\big)(f^{\prime}_y(y_0))^3\cr\noalign{\vskip4truemm}
&\quad \cdot g(y_0)+\mathcal{O}(h^5). \cr\noalign{\vskip4truemm}
\end{aligned}
\end{equation*}
 By comparing with the Taylor expansion of exact solution $y(t_0+h)$, it can be verified that the IMSVERK method with coefficients satisfying the order conditions (\ref{order conditions}) has order four. The proof of the theorem is completed.
 $\hfill \square$
\end{proof}

Theorem \ref{the proof of SVERK24} indicates that there exists a unique fourth-order IMSVERK method with two stages:
 \begin{equation}\label{SVERK24}
\left\{
\begin{array}{l}
\displaystyle
Y_1=e^{-\frac{3-\sqrt{3}}{6}hM}y_0+h\big[\frac{1}{4}f(Y_1)+\frac{3-2\sqrt{3}}{12}f(Y_2)\big],\cr\noalign{\vskip4truemm}
\displaystyle
Y_2=e^{-\frac{3+\sqrt{3}}{6}hM}y_0+h\big[\frac{3+2\sqrt{3}}{12}f(Y_1)+\frac{1}{4}f(Y_2)\big],\cr\noalign{\vskip4truemm}
\displaystyle y_{1}=e^{-hM}y_0+\frac{h}{2}(f(Y_1)+f(Y_2))-\frac{h^2}{2!}Mf(y_0)+\frac{h^3}{3!}\big((M-f'_y(y_0))Mf(y_0)-Mf'_y(y_0)g(y_0)\big) \cr\noalign{\vskip4truemm}
\displaystyle
\quad +\frac{h^4}{4!}\Big((-M+f'_y(y_0))M^2f(y_0)+M^2f'_y(y_0)g(y_0)-Mf''_{yy}(y_0)\big(g(y_0),g(y_0)\big)-Mf'_y(y_0)(-M\cr\noalign{\vskip4truemm}
\displaystyle
\quad +f'_y(y_0))g(y_0)-f'_y(y_0)Mf'_y(y_0)g(y_0)-f'_y(y_0)f'_y(y_0)Mf(y_0)+3f''_{yy}(y_0)\big(-Mf(y_0),g(y_0)\big)\Big).
\end{array}
\right.
\end{equation}
The method (\ref{SVERK24}) can be displayed by the Butcher tableau
\begin{equation}\label{Tableau-SVERK24}
\begin{aligned} &\quad\quad\begin{tabular}{c|c|cc}
 $\frac{1}{2}-\frac{\sqrt{3}}{6}$&$e^{-\frac{3-\sqrt{3}}{6}hM}$&$\frac{1}{4}$&$\frac{3-2\sqrt{3}}{12}$\\[3pt]
  $\frac{1}{2}+\frac{\sqrt{3}}{6}$&$e^{-\frac{3+\sqrt{3}}{6}hM}$&$\frac{3+2\sqrt{3}}{12}$&$\frac{1}{4}$ \\[3pt]
 \hline
  $\raisebox{-1.3ex}[1.0pt]{$e^{-hM}$}$ & $\raisebox{-1.3ex}[1.0pt]{$w_4(z)$}$&$\raisebox{-1.3ex}[1.0pt]{$\frac{1}{2}$}$&$\raisebox{-1.3ex}[1.0pt]{$\frac{1}{2}$}$  \\
\end{tabular}.
\end{aligned}
\end{equation}

The following theorem shows the order conditions of the fourth-order IMSVERK  method with two stages are exactly identical to  (\ref{order conditions}) as well.

\begin{mytheo}\label{MVERK4proof}
Suppose that the coefficients of a two-stage IMMVERK method with $\bar{w}_4(z)$
 \begin{equation}\label{MVERK2-4}
\left\{
\begin{array}{l}
\displaystyle
\bar{Y}_1=y_0+h\big[\bar{a}_{11}(-M\bar{Y}_1+f(\bar{Y}_1))+\bar{a}_{12}(-M\bar{Y}_2+f(\bar{Y}_2))\big],\cr\noalign{\vskip4truemm}
\bar{Y}_2=y_0+h\big[\bar{a}_{21}(-M\bar{Y}_1+f(\bar{Y}_1))+\bar{a}_{22}(-M\bar{Y}_2+f(\bar{Y}_2))\big],\cr\noalign{\vskip4truemm}
\displaystyle \bar{y}_{1}=e^{-hM}y_0+h(\bar{b}_1f(\bar{Y}_1)+\bar{b}_2f(\bar{Y}_2))+\bar{w}_4(z),
\end{array}
\right.
\end{equation}
where
\begin{equation*}
\begin{aligned}
\bar{w}_4(z)&=-\frac{h^2}{2!}Mf(y_0)+\frac{h^3}{3!}(M^2f(y_0)-Mf'_y(y_0)g(y_0)) +\frac{h^4}{4!}(-M^3f(y_0)+M^2f'_y(y_0)g(y_0)-Mf''_{yy}(y_0)\cr\noalign{\vskip4truemm}
&\quad \cdot (g(y_0),g(y_0)) -Mf'_y(y_0)(-M+f'_y(y_0))g(y_0)),
\end{aligned}
\end{equation*}
which satisfy (\ref{order conditions}), then the IMMVERK method has order four.
\end{mytheo}

\begin{proof}
Under the assumptions $\bar{c}_1=\bar{a}_{11}+\bar{a}_{12}$ and $\bar{c}_2=\bar{a}_{21}+\bar{a}_{22}$, by comparing  the Taylor expansions   of numerical solution $\bar{y}_1$ with exact solution $y(t_0+h)$, the result can  be proved in a same way as Theorem \ref{the proof of SVERK24}. Therefore, we omit the details.
$\hfill\square$
\end{proof}

Then, we present the unique fourth-order IMMVERK  method with two stages
 \begin{equation}\label{MVERK24}
\left\{
\begin{array}{l}
\displaystyle
Y_1=y_0+h\Big[\frac{1}{4}(-MY_1+f(Y_1))+\frac{3-2\sqrt{3}}{12}(-MY_2+f(Y_2))\Big],\cr\noalign{\vskip4truemm}
\displaystyle
Y_2=y_0+h\Big[\frac{3+2\sqrt{3}}{12}(-MY_1+f(Y_1))+\frac{1}{4}(-MY_2+f(Y_2))\Big],\cr\noalign{\vskip4truemm}
\displaystyle y_{1}=e^{-hM}y_0+\frac{h}{2}(f(Y_1)+f(Y_2))-\frac{h^2}{2!}Mf(y_0)+\frac{h^3}{3!}(M^2f(y_0)-Mf'_y(y_0)g(y_0))
 +\frac{h^4}{4!}(-M^3f(y_0) \cr\noalign{\vskip4truemm}
 \displaystyle \qquad +M^2f'_y(y_0)g(y_0)-Mf''_{yy}(y_0)(g(y_0),g(y_0))
 -Mf'_y(y_0)(-M+f'_y(y_0))g(y_0)),
\end{array}
\right.
\end{equation}
with $g(y_0)=-My_0+f(y_0)$, which can be denoted by the Butcher tableau
\begin{equation}\label{Tableau-MVERK24}
\begin{aligned} &\quad\quad\begin{tabular}{c|c|cc}
 $\frac{1}{2}-\frac{\sqrt{3}}{6}$&$I$&$\frac{1}{4}$&$\frac{3-2\sqrt{3}}{12}$\\[3pt]
  $\frac{1}{2}+\frac{\sqrt{3}}{6}$&$I$&$\frac{3+2\sqrt{3}}{12}$&$\frac{1}{4}$ \\[3pt]
 \hline
  $\raisebox{-1.3ex}[1.0pt]{$e^{-hM}$}$ & $\raisebox{-1.3ex}[1.0pt]{$\bar{w}_4(z)$}$&$\raisebox{-1.3ex}[1.0pt]{$\frac{1}{2}$}$&$\raisebox{-1.3ex}[1.0pt]{$\frac{1}{2}$}$  \\
\end{tabular}.
\end{aligned}
\end{equation}

In view of  \eqref{SVERK24} and \eqref{MVERK24},  two methods  reduce to the classical implicit  fourth-order  RK method with two stages by Hammer and Hollingsworth (see, e.g. \cite{Hammer1955}) once $M\rightarrow \mathbf{0}$.  It should be noted that  \eqref{SVERK24} and \eqref{MVERK24} use the Jacobian
matrix and Hessian matrix of $f(y)$ with respect to $y$ at each step, however, to our knowledge, the idea for stiff
problems is no by means new (see, e.g. \cite{Abhulimen2014,Cash1981Siam,Enright1974siam}). Using the Gaussian points $c_1=\frac{1}{2}-\frac{\sqrt{3}}{6}$ and $c_2=\frac{1}{2}+\frac{\sqrt{3}}{6}$, the ERK method of collocation type with two stages has been formulated by Hochbruck et al. \cite{Hochbruck2005b}, which can be represented by the Butcher tableau
\begin{equation}\label{Tableau-ERK24}
\begin{aligned} &\quad\quad\begin{tabular}{c|cc}
 $\frac{1}{2}-\frac{\sqrt{3}}{6}$&$\frac{\sqrt{3}}{6}\varphi_1(-c_1hM)-\sqrt{3}c_1^2\varphi_2(-c_1hM)$&$-\sqrt{3}c_1^2\varphi_1(-c_1hM)+\sqrt{3}c_1^2\varphi_2(-c_1hM)$\\[2pt]
  $\frac{1}{2}+\frac{\sqrt{3}}{6}$&$\sqrt{3}c_2^2\varphi_1(-c_2hM)-\sqrt{3}c_2^2\varphi_2(-c_2hM)$& $-\frac{\sqrt{3}}{6}\varphi_1(-c_2hM)+\sqrt{3}c_2^2\varphi_2(-c_2hM)$ \\[2pt]
 \hline
 &$\raisebox{-1.3ex}[1.0pt]{$\sqrt{3}c_2\varphi_1(-hM)-\sqrt{3}\varphi_2(-hM)$}$&$\raisebox{-1.3ex}[1.0pt]{$-\sqrt{3}c_1\varphi_1(-hM)+\sqrt{3}\varphi_2(-hM)$}$  \\
\end{tabular}.
\end{aligned}
\end{equation}
 It is clear that the coefficients of (\ref{Tableau-ERK24}) are  matrix-valued functions $\varphi_k(-hM)$ $(k>0)$, their implementation depends on  the computing or approximating the product of a matrix exponential function
with a vector. As we consider the variable stepsize technique, the coefficients of these exponential integrators  are needed to recalculate at each step.   In contrast,  the coefficients of SVERK and MVERK methods are the real constants, their operations can  reduce the computation of matrix exponentials to some extent. On the other hand, compared with standard RK methods,  these ERK methods possess some properties of matrix exponentials and exactly integrate the homogeneous linear equation \eqref{homogeneous}.

\section{Linear stability}\label{sec4}

In what follows we investigate the linear stability properties of   IMMVERK and IMSVERK methods. For classical RK methods, the linear stability  analysis  is related to the Dahlquist test equation \cite{Hairer2006}
$$y^{\prime}=\lambda y,\ \lambda \in \mathbb R.$$   When we consider the exponential integrators,  the stability properties of an exponential method  are analyzed by  applying the method to the partitioned Dalquist equation \cite{Buvoli2022}
\begin{equation}\label{test equation}
y^{\prime}=i\lambda_1y+i\lambda_2y, \qquad y(t_0)=y_0, \qquad \qquad \lambda_1, \ \lambda_2 \in \mathbb R.
\end{equation}
i.e.,  solving (\ref{test equation}) by a partitioned exponential integrator, and  treating the $i\lambda_1$
exponentially and the $i\lambda_2$ explicitly, then we have the explicit  scalar form
\begin{equation}\label{stability function}
y_{n+1}=R(ik_1,ik_2)y_n, \quad \quad \quad k_1=h\lambda_1, \ k_2=h\lambda_2.
\end{equation}
\begin{defn} For an exponential method with $R(ik_1,ik_2)$ given by \eqref{stability function},
the set $S$ of the stability function $R(ik_1,ik_2)$
$$S=\big\{(k_1,k_2)\in \mathbb R^2 : |R(ik_1,ik_2)|\leq1 \big \},$$
 is called the stability region of  the exponential method.
\end{defn}

Applying  the IMMVERK  method \eqref{MVERK22}   and  IMSVERK method  \eqref{SVERK21}  to \eqref{test equation}, the  stability regions of   implicit second-order ERK methods  are respectively depicted in Fig. \ref{second} (b) and (c), and we also plot the stability region of the explicit second-order MVERK method  with two stages  (see, e.g. \cite{Wang2022}) in Fig. \ref{second} (a). For fourth-order ERK methods, we select the IMMVERK method \eqref{MVERK24}, IMSVERK  method \eqref{SVERK24}  and the explicit SVERK method with four stages (see, e.g. \cite{Hu2022}) to make a comparison,  the  stability regions of these methods are shown in Fig. \ref{fourth}. It can be observed that  the implicit exponential methods possess the comparable stability regions than explicit exponential methods.

 \begin{figure}[!htb]
\centering
\begin{tabular}[c]{cccc}%
  \subfigure[]{\includegraphics[width=5.2cm,height=4.7cm]{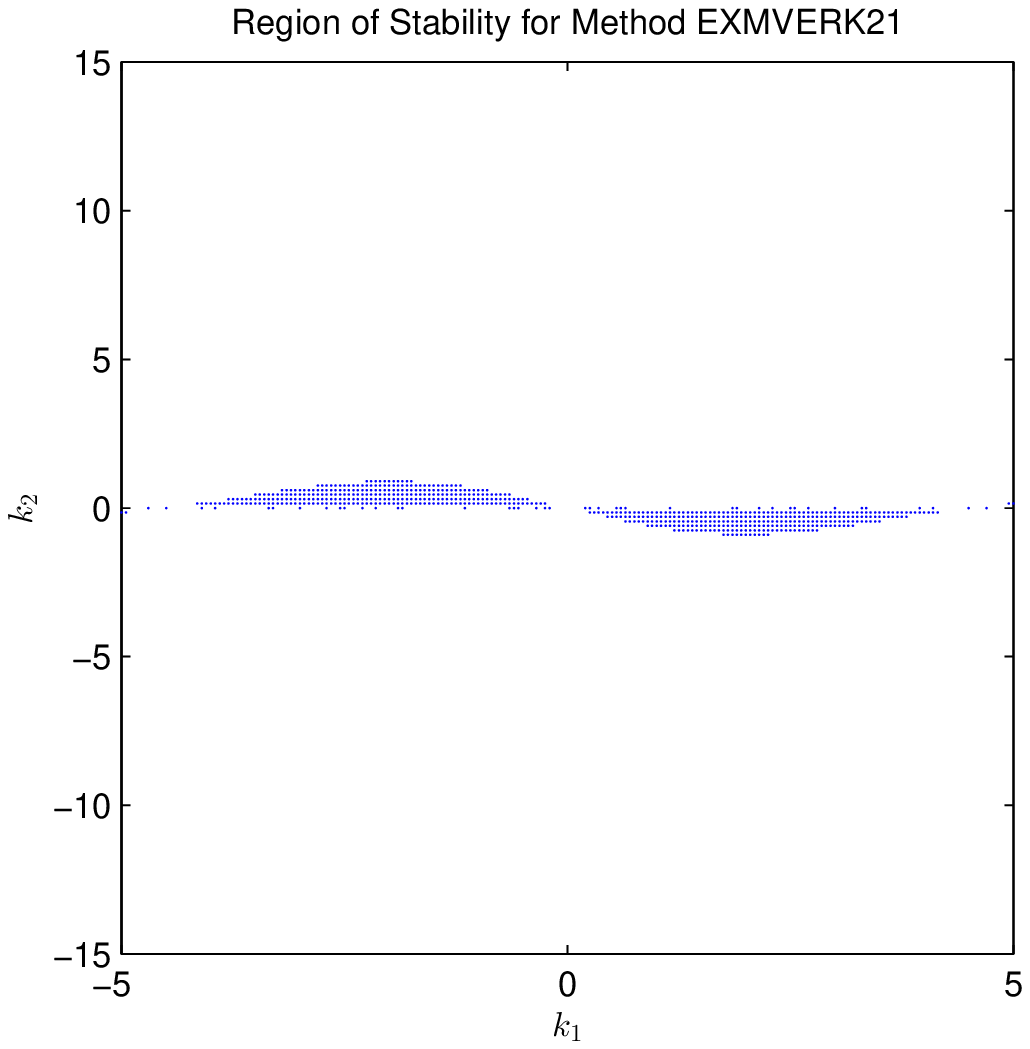}}
  \subfigure[]{\includegraphics[width=5.2cm,height=4.7cm]{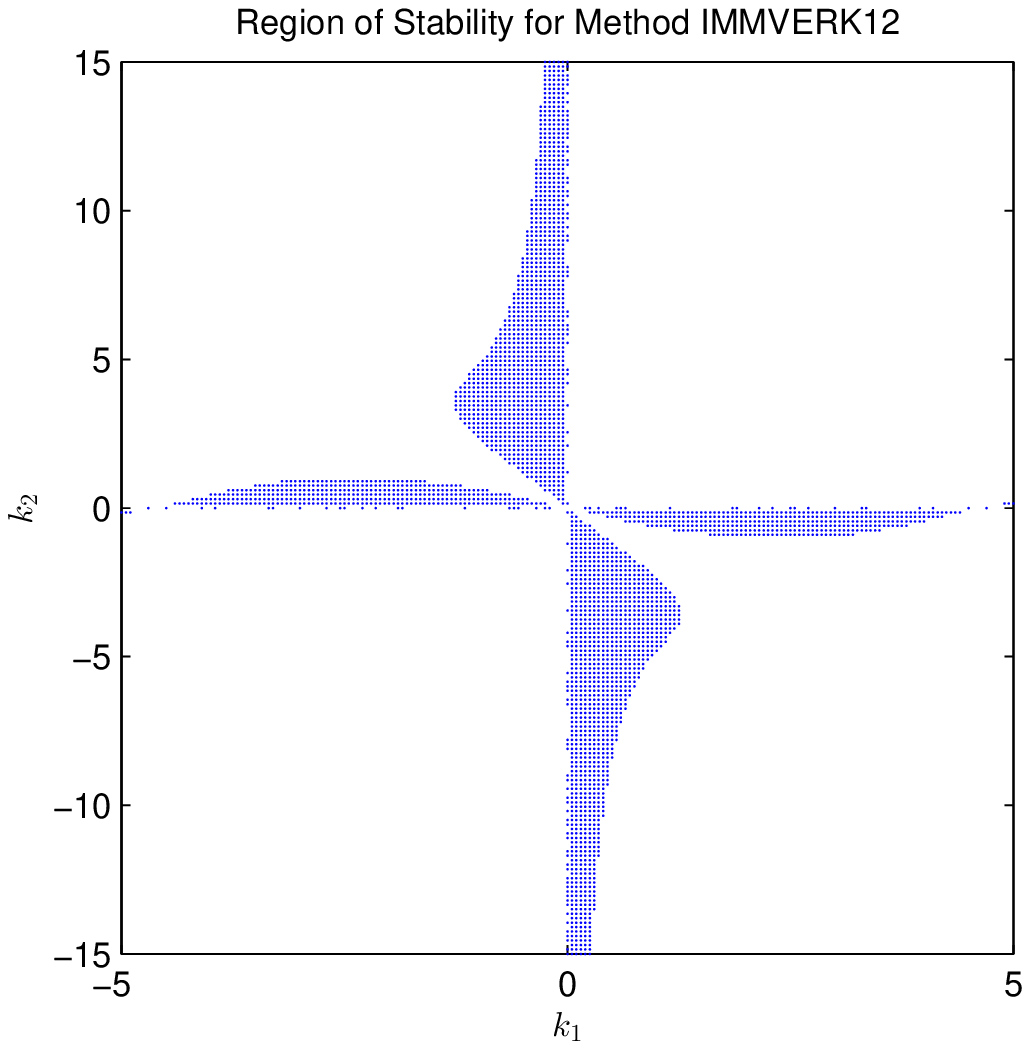}}
  \subfigure[]{\includegraphics[width=5.2cm,height=4.7cm]{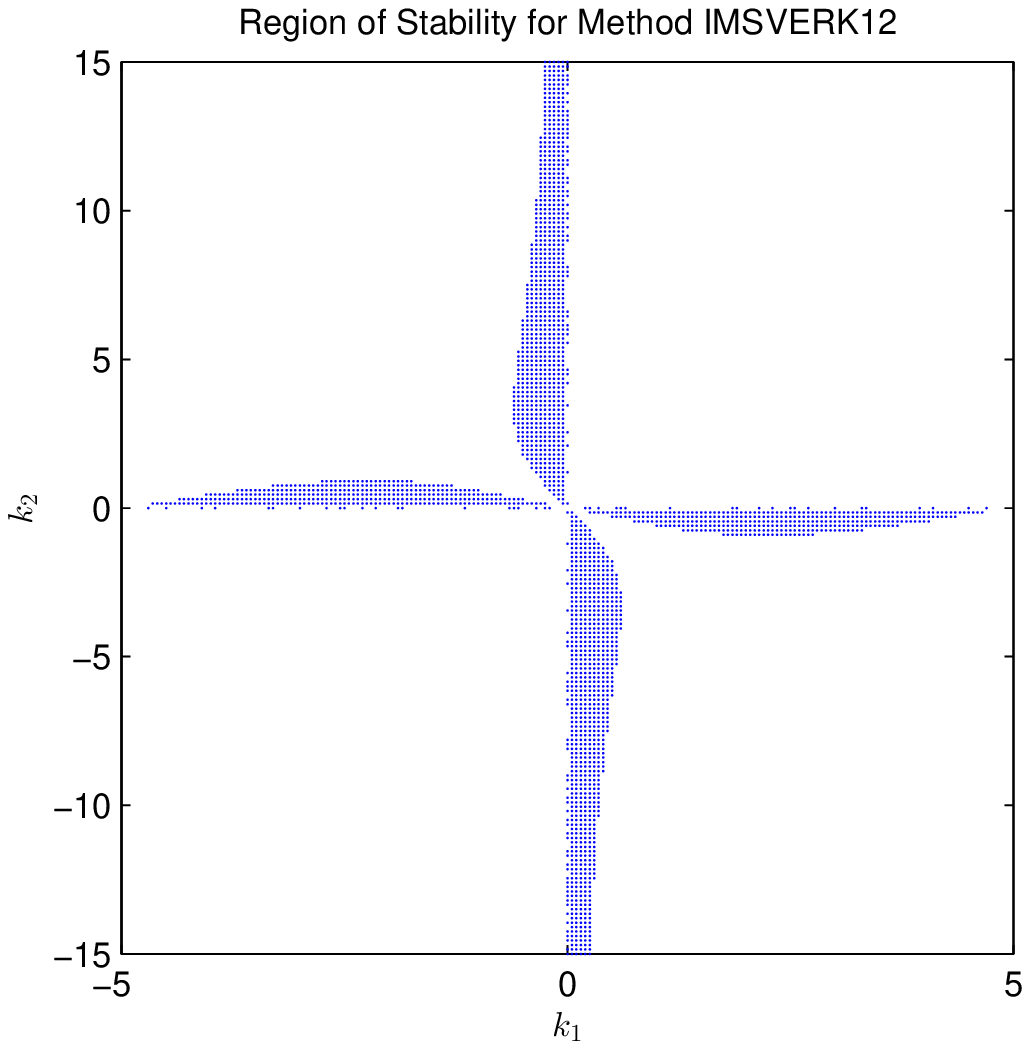}}
\end{tabular}
\caption{ {\bf {(a)}}: Stability region for explicit second-order  MVERK (EXMVERK21) method with two stages. {\bf {(b)}} Stability region for implicit second-order
 MVERK (IMMVERK12) method.  {\bf {(c)}}: Stability region for implicit second-order  SVERK (IMSVERK12) method. }\label{second}
\end{figure}

 \begin{figure}[!htb]
\centering
\begin{tabular}[c]{cccc}%
  \subfigure[]{\includegraphics[width=5.2cm,height=4.7cm]{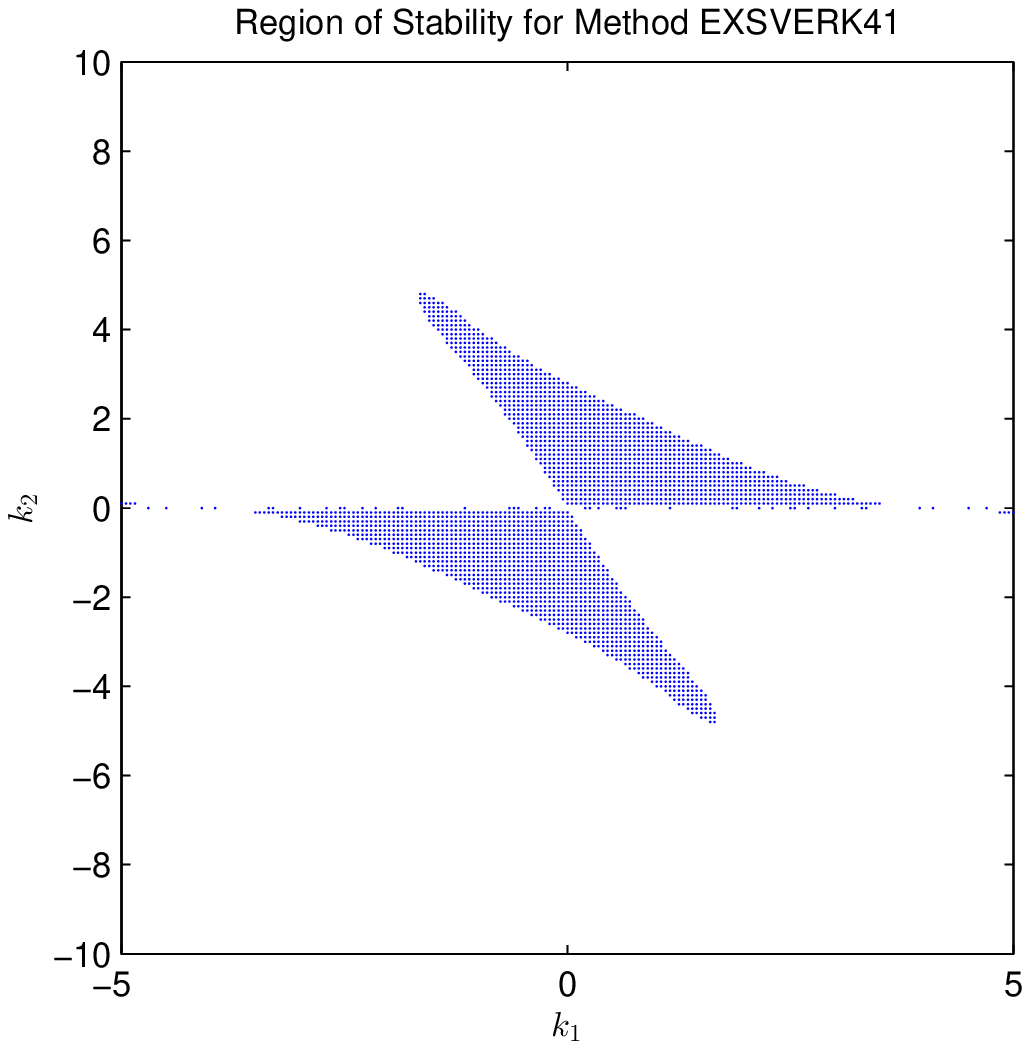}}
  \subfigure[]{\includegraphics[width=5.2cm,height=4.7cm]{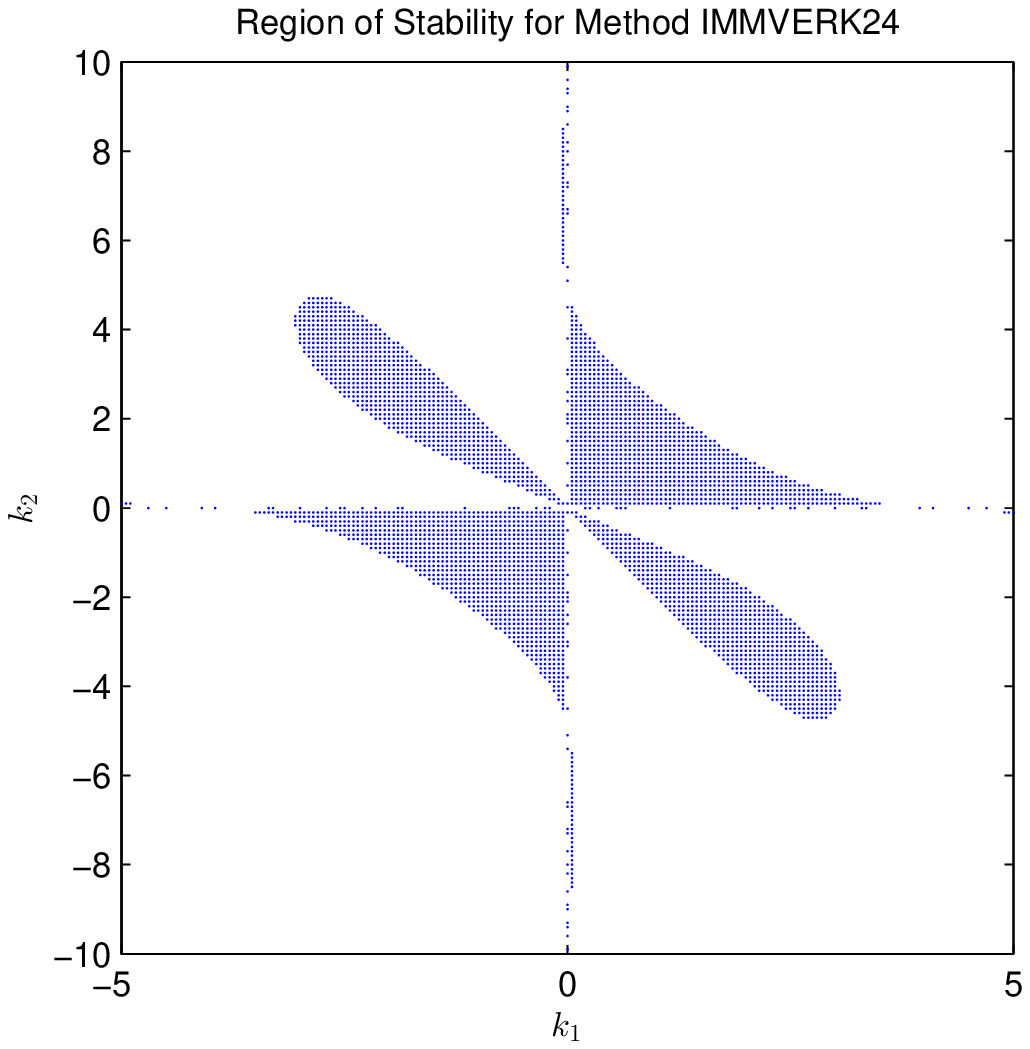}}
  \subfigure[]{\includegraphics[width=5.2cm,height=4.7cm]{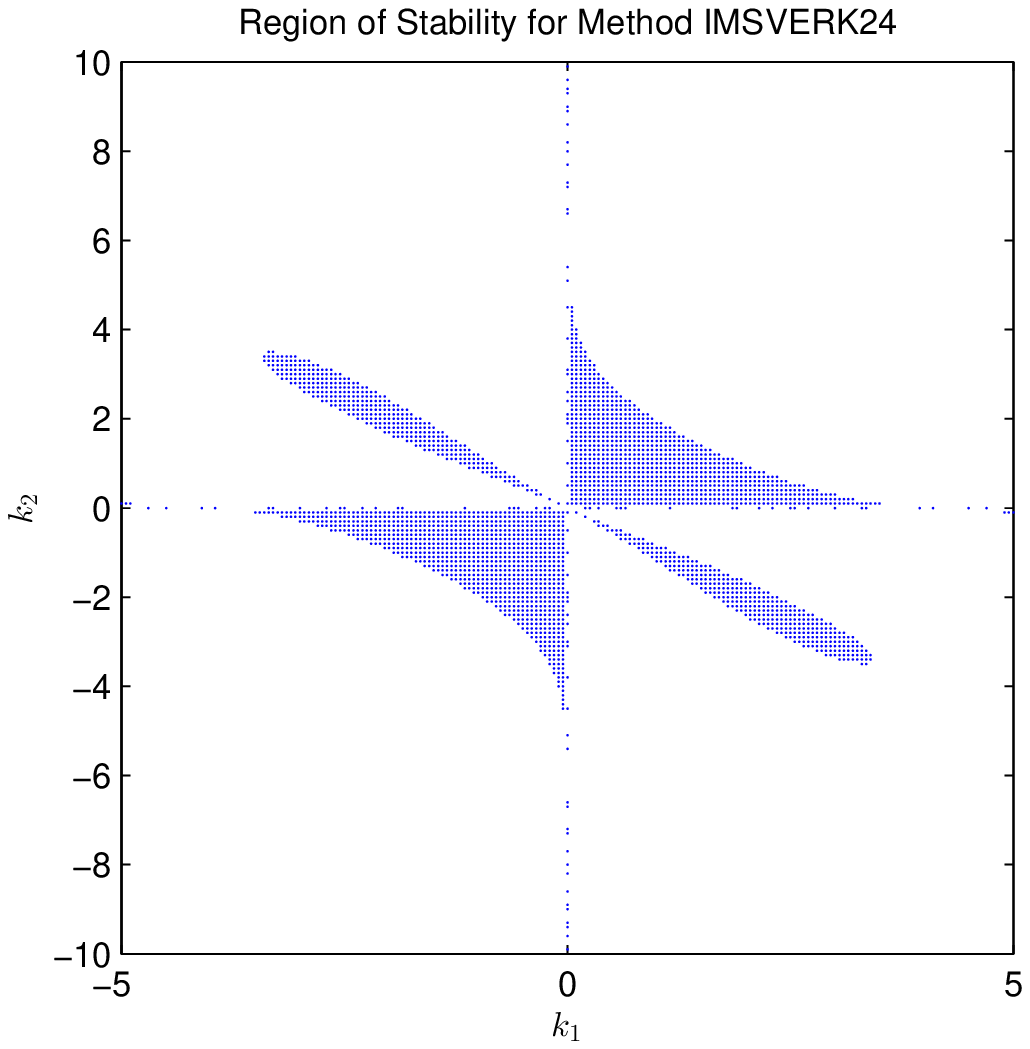}}
\end{tabular}
\caption{ {\bf {(a)}}: Stability region for explicit  fourth-order  SVERK (EXSVERK41) method with four stages. {\bf {(b)}}
 Stability region for  implicit fourth-order  MVERK (IMMVERK24) method.  {\bf {(c)}}: Stability region for implicit fourth-order  SVERK (IMSVERK24) method.}\label{fourth}
\end{figure}

\section{Numerical Experiments}\label{sec5}

 In this section, we respectively apply the symplectic method (\ref{symplectic SVERK1-1}) and  highly accurate IMSVERK and IMMVERK methods to solve some highly
 oscillatory problems. In order to illustrate the accuracy and efficiency of these exponential methods, we select the standard implicit ERK methods to make comparisons. Since these methods are implicit, we use the fixed-point iteration, and
 the iteration will be stopped once the norm of the difference of two successive approximations is smaller than $10^{-14}$.  The matrix-valued functions $\varphi_k(-hM)$ $(k>0)$ are evaluated by the Krylov subspace method \cite{Berland2007}, which possesses the fast convergence.  In all numerical experiments, we take the Eucildean norm for global errors (GE) and denote the global error of Hamiltonian by GEH.   The following numerical methods are chosen for comparison:

\begin{itemize}   \item First-order   methods: \begin{itemize}
  \item IMSVERK1: the  1-stage implicit SVERK   method \eqref{symplectic SVERK1-1}  of order (non-stiff) one presented in this paper;
  \item EEuler: the  1-stage explicit exponential Euler method  of order (stiff) one proposed in \cite{Hochbruck2005a};
    \item IMEEuler: the  1-stage implicit exponential Euler  method of order (stiff) one proposed in \cite{Hochbruck2005b}.
\end{itemize}
\item Second-order methods:
\begin{itemize}
  \item IMMVERK12: the  1-stage implicit MVERK  method \eqref{MVERK22}  of order (non-stiff) two presented in this paper;
      \item IMSVERK12: the  1-stage implicit SVERK  method  \eqref{SVERK21}  of order (non-stiff) two presented in this paper;
        \item IMERK12: the  1-stage implicit ERK  method \eqref{Tableau-ERK12} of order (stiff) two proposed in \cite{Hochbruck2005b}.
\end{itemize}
\item Fourth-order methods:
\begin{itemize}
  \item IMMVERK24: the 2-stage implicit MVERK method \eqref{MVERK24} of order (non-stiff) four presented in this paper;
      \item IMSVERK24: the 2-stage implicit SVERK method  \eqref{SVERK24}  of order (non-stiff) four presented in this paper;
        \item IMERK24: the  2-stage implicit ERK  method \eqref{Tableau-ERK24}   proposed in \cite{Hochbruck2005b}.
\end{itemize}

\end{itemize}

\begin{problem}\label{Henon}
 The H\'{e}non-Heiles Model  is used to describe the stellar motion (see, e.g. \cite{Hairer2006}), which has the following identical form
\begin{equation*}
\begin{aligned}& \left(
                   \begin{array}{c}
                     x_1 \\
                      x_2 \\
                      y_1 \\
                      y_2\\
                   \end{array}
                 \right)
'+\left(
    \begin{array}{cccc}
      0 & 0 & -1 &0\\
      0 & 0 & 0  &-1\\
      1 & 0 & 0  &0 \\
      0 & 1 & 0  &0\\
    \end{array}
  \right)\left(
                   \begin{array}{c}
                     x_1 \\
                      x_2 \\
                      y_1\\
                      y_2\\
                   \end{array}
                 \right)=
\left(
  \begin{array}{c}
  0\\
  0\\
  -2x_1x_2\\
  -x_1^2+x_2^2\\
                                                                           \end{array}
                                                                         \right).
\end{aligned}\end{equation*}
The Hamiltonian  of this system is given by
 \begin{equation*}
 H(x,y)=\frac{1}{2}(y_1^2+y_2^2)+\frac{1}{2}(x_1^2+x_2^2)+x_1^2x_2-\frac{1}{3}x_2^3.
 \end{equation*}
 We select the initial values as
 \begin{equation*}
 \big(x_1(0),x_2(0),y_1(0),y_2(0)\big)^{\intercal}=(\sqrt{\frac{11}{96}},0,0,\frac{1}{4})^{\intercal}.
 \end{equation*}
 At first, Fig. \ref{HHfirst} (a) presents  the problem is solved on the interval $[0,10]$  with stepsizes $h=1/2^{k},\ k=2,\ldots,6$ for IMSVERK1s1, EEuler, IMEEuler, it can be observed that the IMSVERK1s1 method has higher accuracy than first-order exponential Euler methods.  We integrate this problem  with  stepsize  $h=1/30$ on the interval $[0,100]$, the  relative errors $RGEH=\frac{GEH}{H_0}$ of Hamiltonian energy   for IMSVERK1s1, EEuler and IMEEuler are presented by  Fig. \ref{HHfirst} (b) and Fig. \ref{HHenergy}, which reveals the structure-preserving properties of  the symplectic method.   Finally, we integrate  this system  over the interval $[0,10]$ with
 stepsizes $h=1/2^{k}$ for $k=2,\ldots,6$, Figs. \ref{HHsecond} and \ref{HHfourth} display the global errors (GE) against the stepszies and the CPU time (seconds) for IMMVERK12, IMSVERK12, IMERK12, IMMVERK24,  IMSVERK24,  IMERK24. We observe that our methods present the comparable accuracy and efficiency.
\begin{figure}[!htb]
\centering
\begin{tabular}[c]{cccc}%
  \subfigure[]{\includegraphics[width=6cm,height=5.4cm]{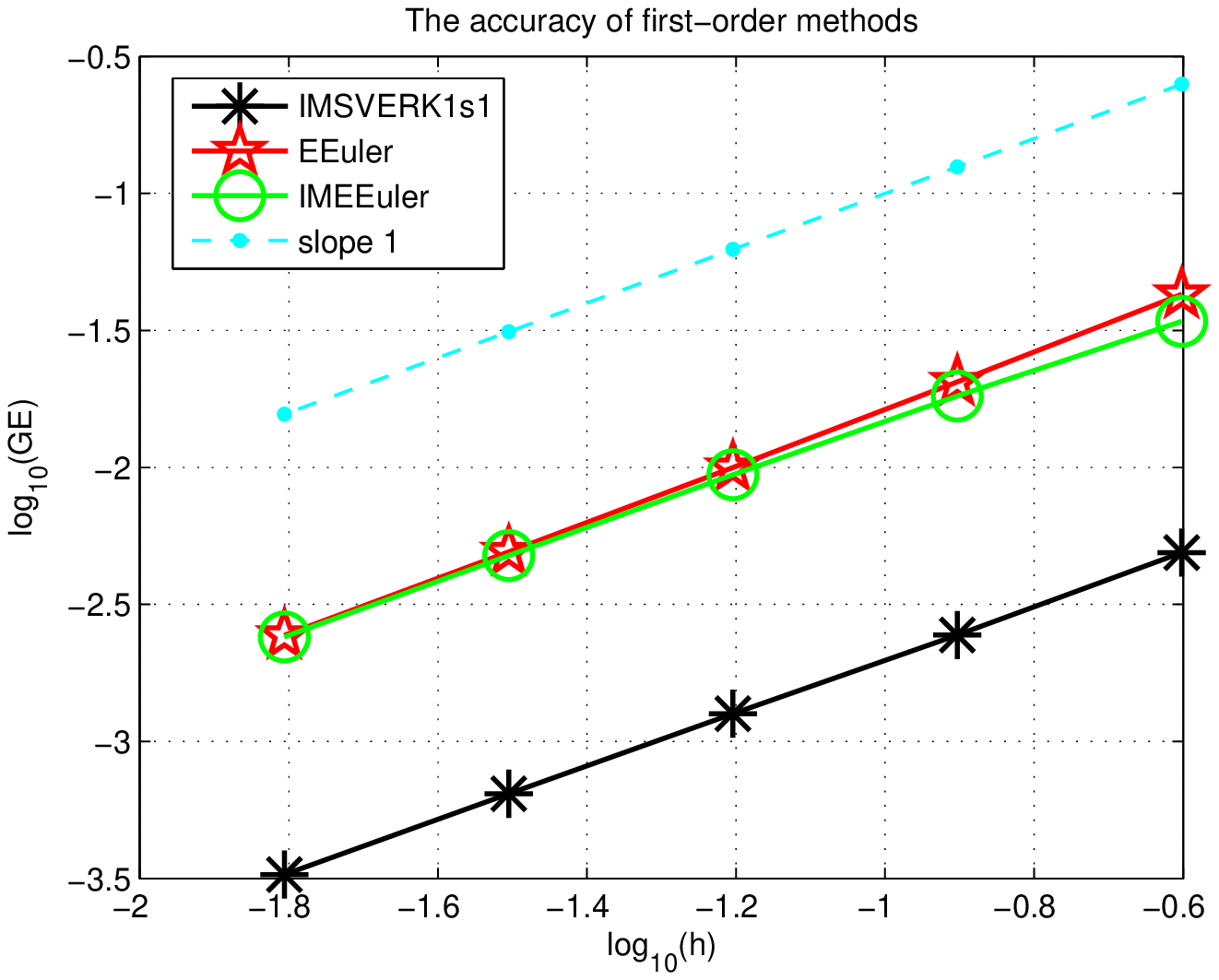}}
  \subfigure[]{\includegraphics[width=6cm,height=5.4cm]{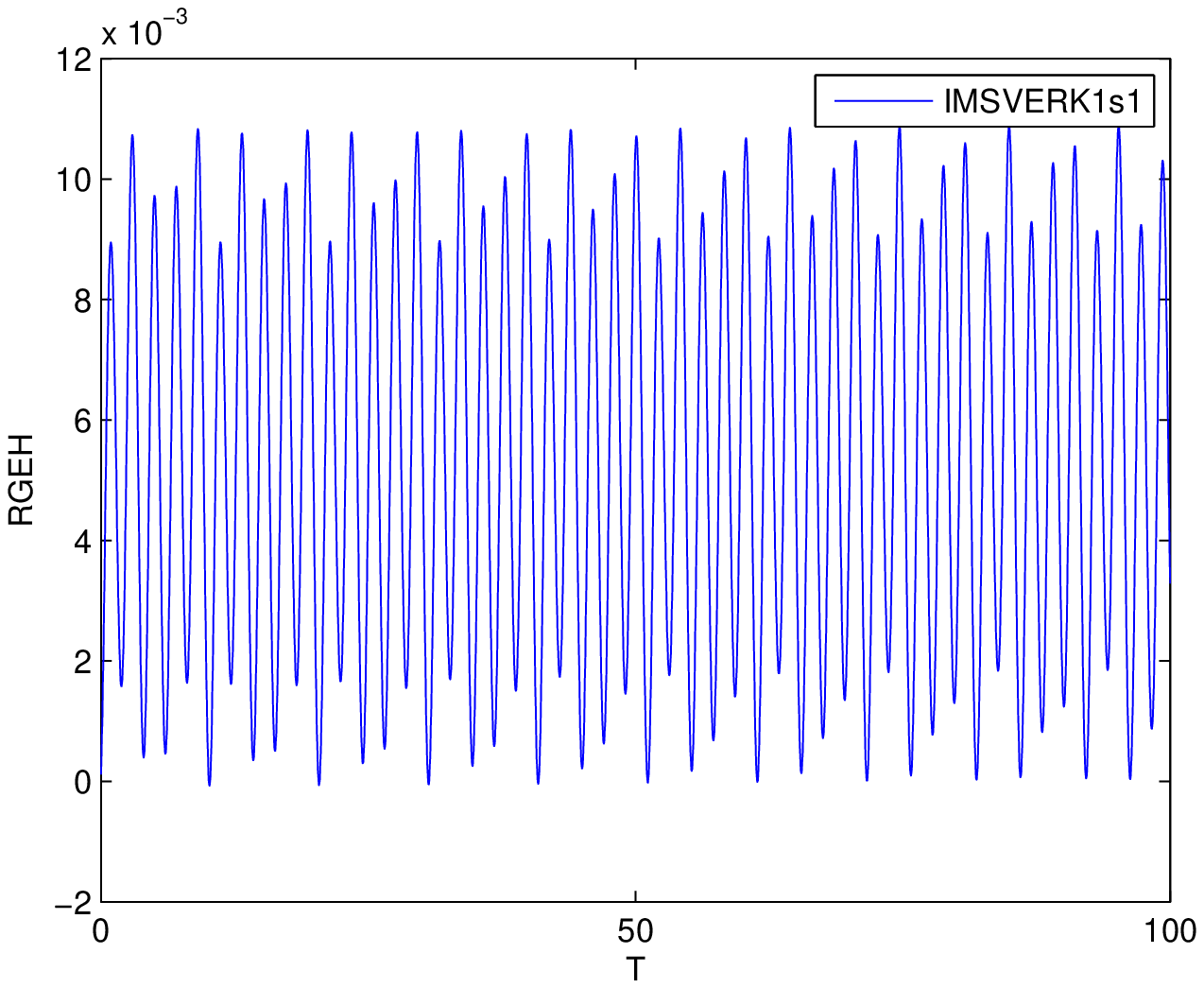}}
\end{tabular}
\caption{Results for  Problem \ref{Henon}. {\bf {(a)}}: The $\log$-$\log$ plots of global
errors (GE) against $h$. {\bf {(b)}}: the energy preservation for method IMSVERK1s1.}\label{HHfirst}
\end{figure}

\begin{figure}[!htb]
\centering
\begin{tabular}[c]{cccc}%
  \subfigure[]{\includegraphics[width=6cm,height=5.7cm]{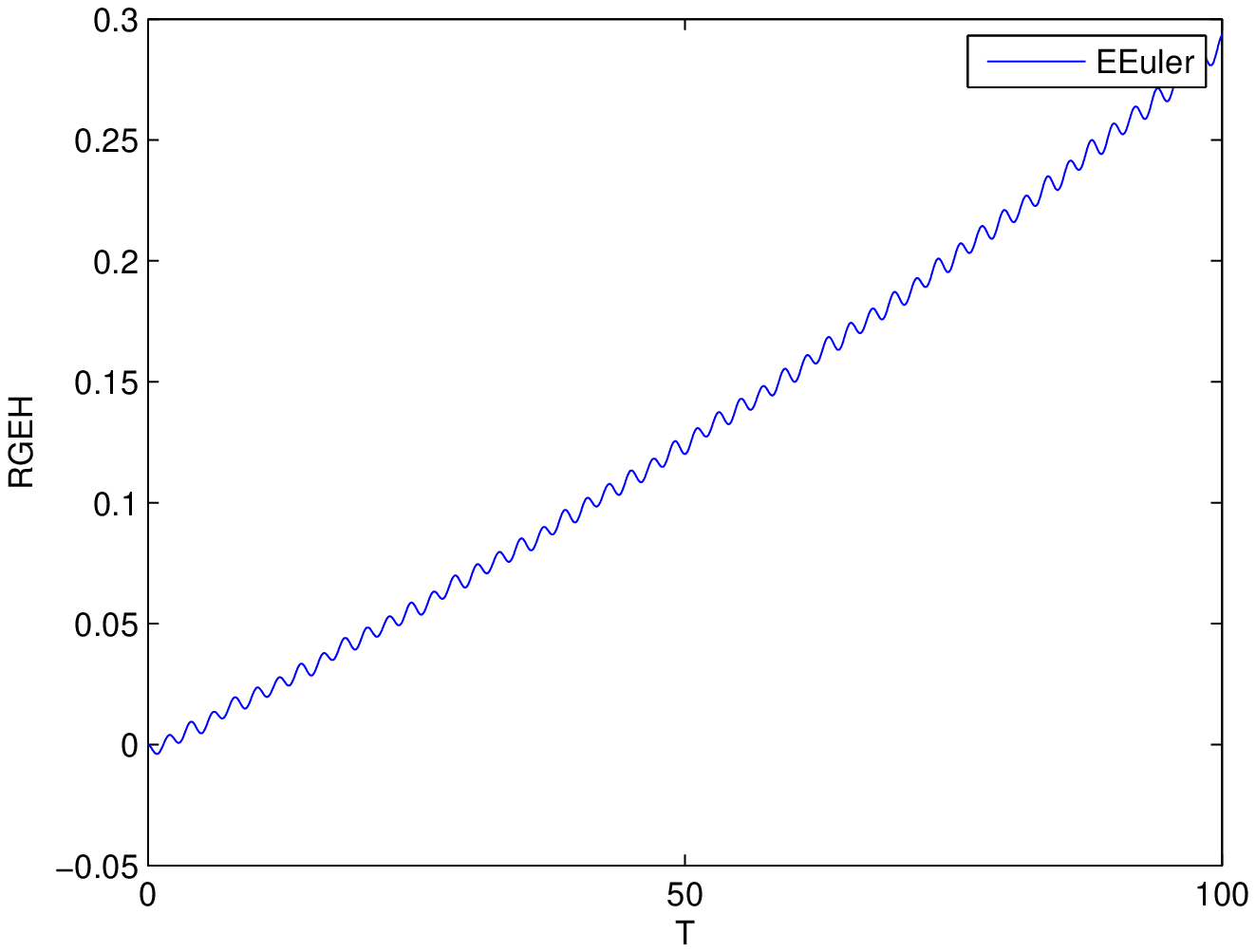}}
  \subfigure[]{\includegraphics[width=6cm,height=5.7cm]{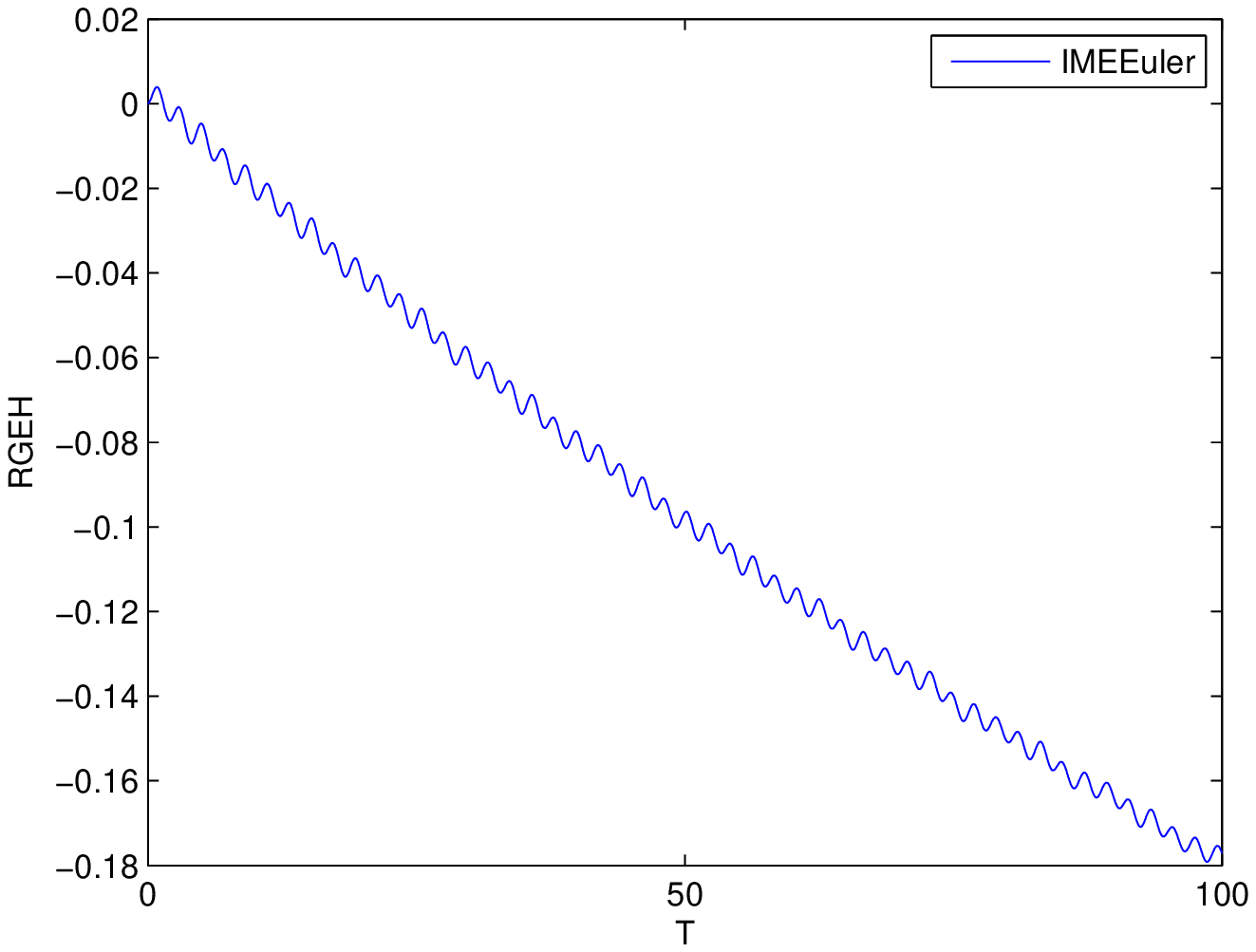}}
\end{tabular}
\caption{Results for  Problem \ref{Henon}. : the energy preservation for methods EEuler {\bf {(a)}} and IMEEuler {\bf {(b)}}.}\label{HHenergy}
\end{figure}

\begin{figure}[!htb]
\centering
\begin{tabular}[c]{cccc}%
  \subfigure[]{\includegraphics[width=6cm,height=5.7cm]{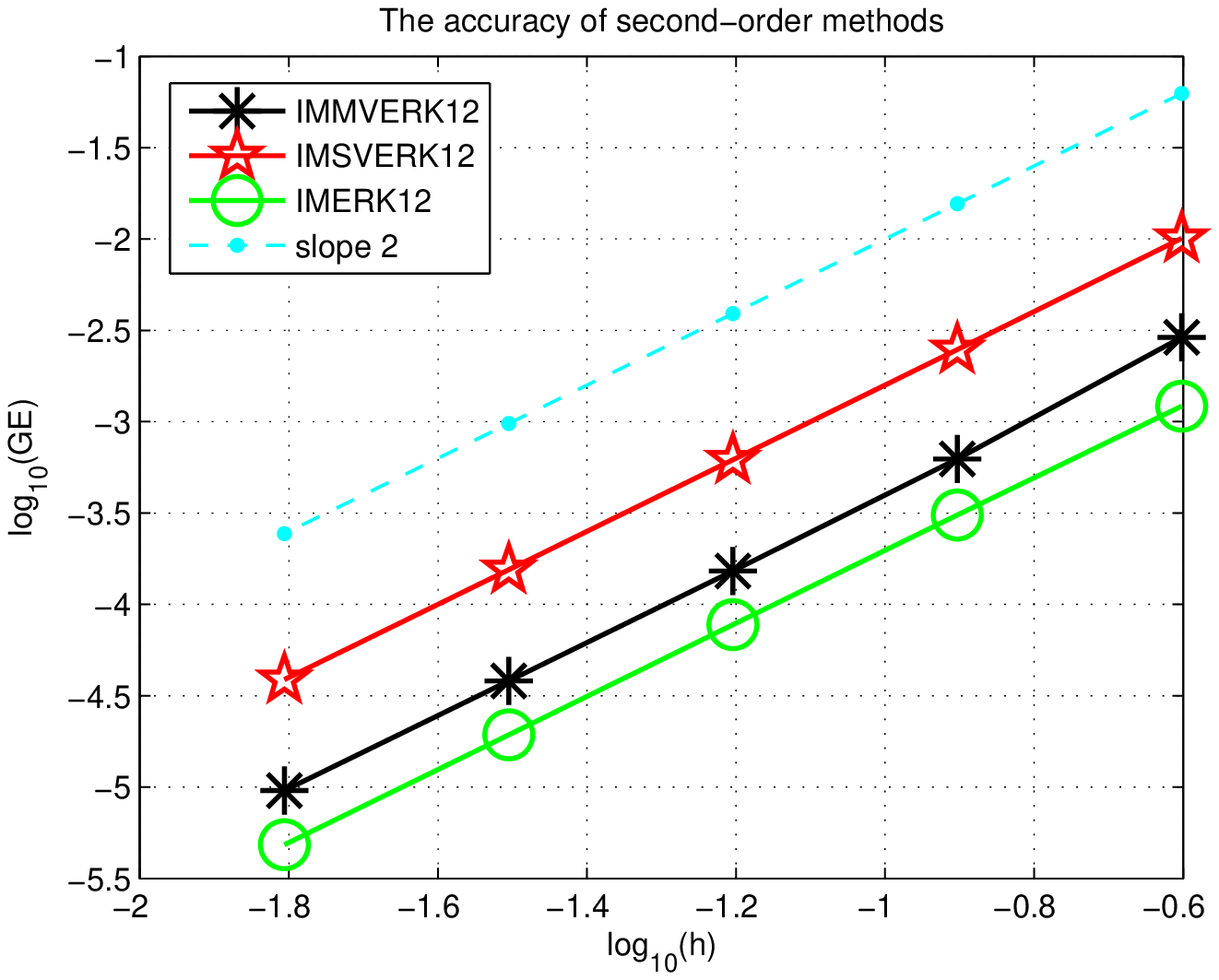}}
  \subfigure[]{\includegraphics[width=6cm,height=5.7cm]{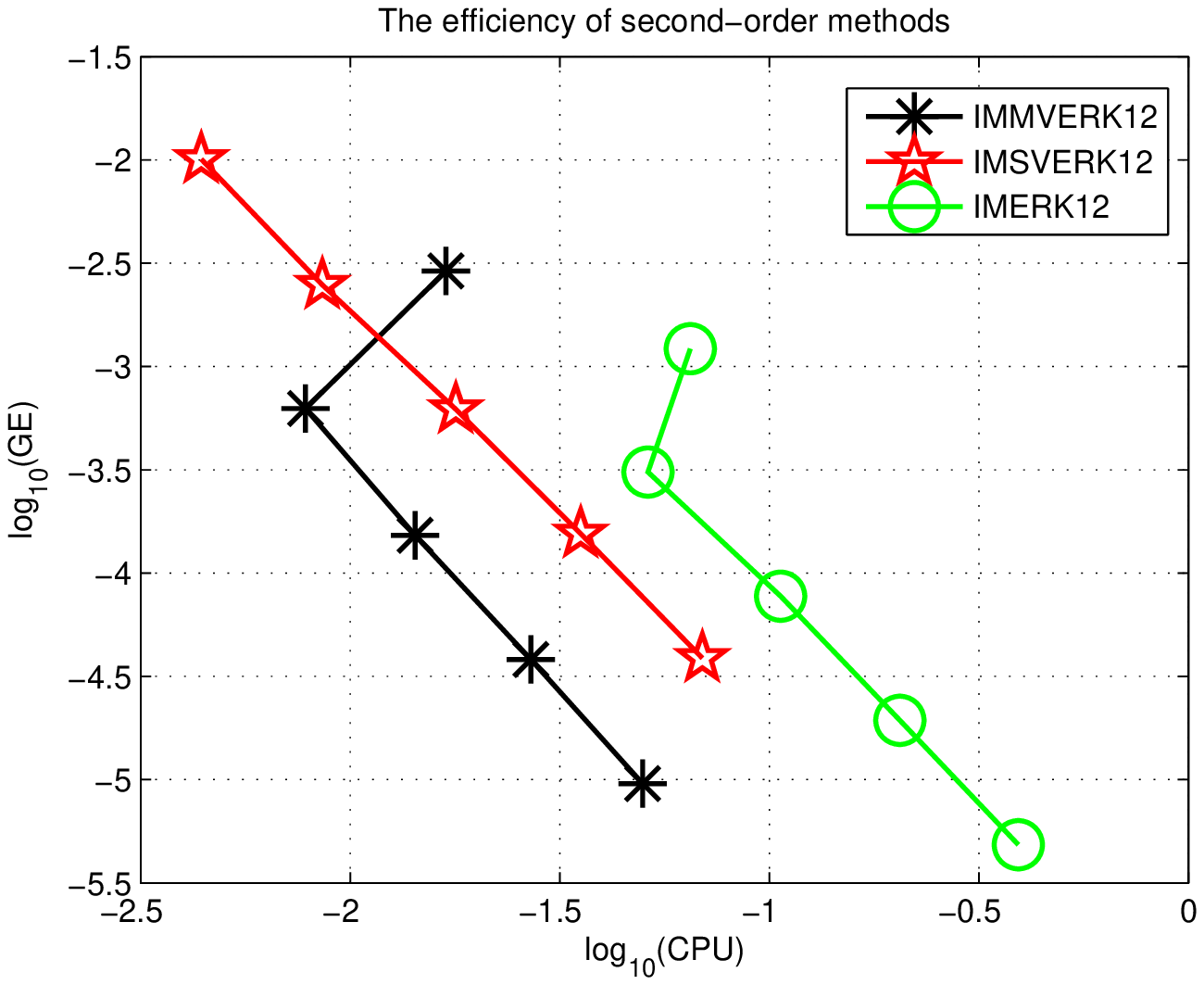}}
\end{tabular}
\caption{Results for  Problem \ref{Henon}. {\bf {(a)}}: The $\log$-$\log$ plots of global
errors (GE) against $h$. {\bf {(b)}}: The $\log$-$\log$ plots of global
errors against the CPU time.}\label{HHsecond}
\end{figure}

\begin{figure}[!htb]
\centering
\begin{tabular}[c]{cccc}%
  \subfigure[]{\includegraphics[width=6cm,height=5.7cm]{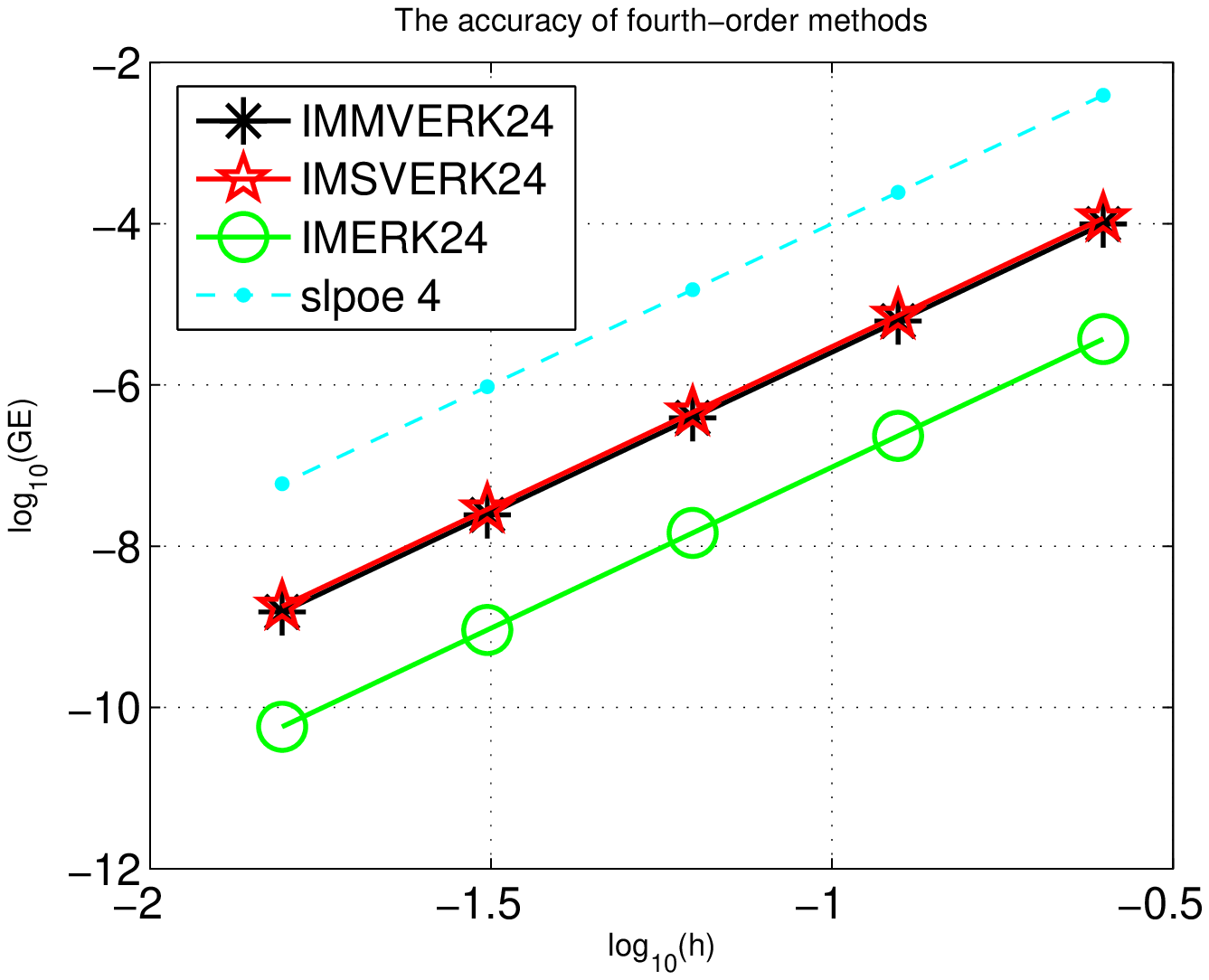}}
  \subfigure[]{\includegraphics[width=6cm,height=5.7cm]{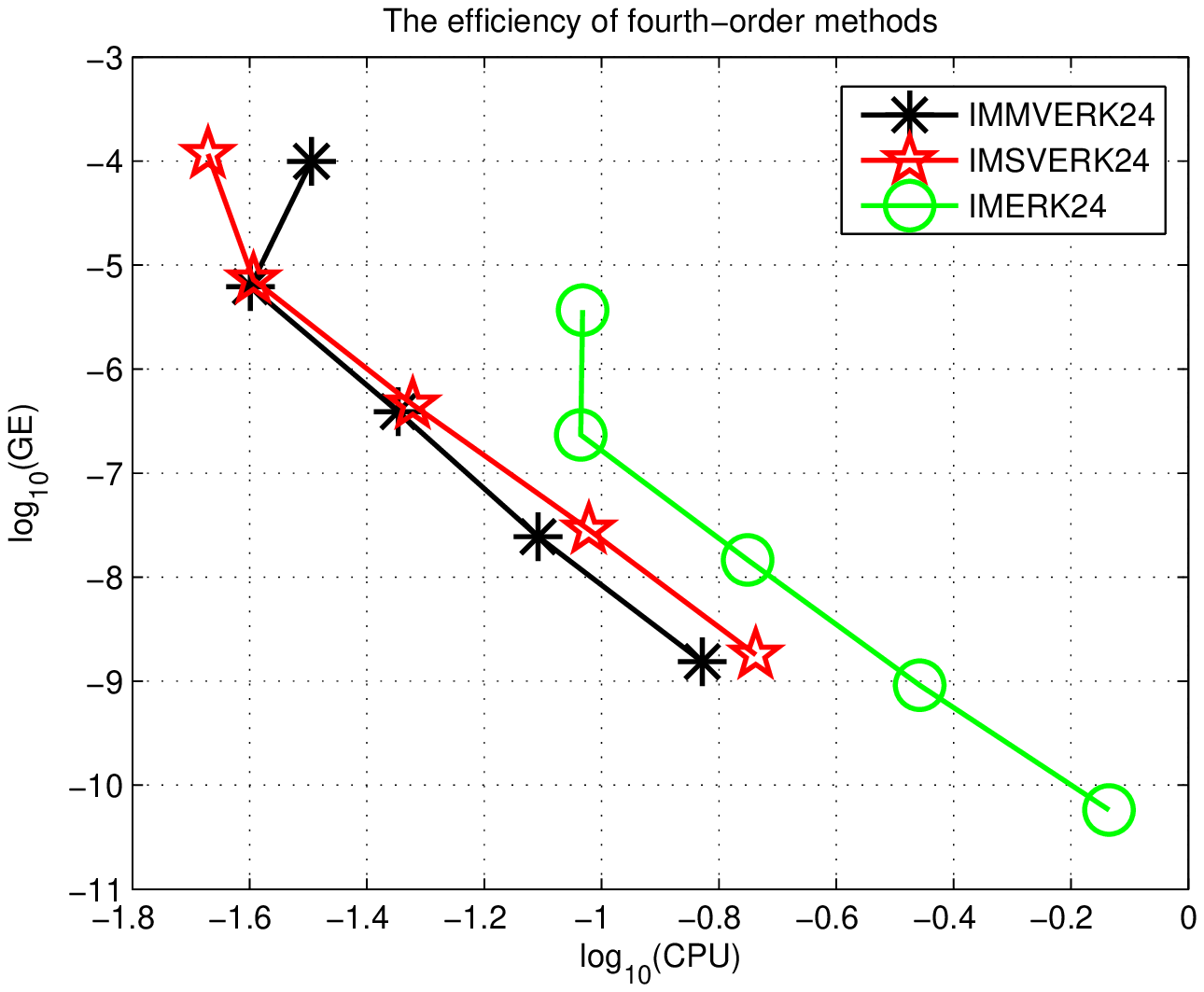}}
\end{tabular}
\caption{Results for  Problem \ref{Henon}. {\bf {(a)}}: The $\log$-$\log$ plots of global
errors (GE) against $h$. {\bf {(b)}}: The $\log$-$\log$ plots of global
errors against the CPU time.}\label{HHfourth}
\end{figure}
\end{problem}

\begin{problem} \label{Duffing} Consider the Duffing equation \cite{Mei2017}
\begin{equation}\left\{
\begin{array}{c}
\ddot{q}+\omega^2q=k^2(2q^3-q), \cr\noalign{\vskip4truemm} q(0)=0, \ \dot{q}(0)=\omega,
\end{array}
\right.
\end{equation}
where $0\leq k <\omega$.
\end{problem}

Set $p=\dot{q}$, $z=(p,q)^{\intercal}$. We rewrite the Duffing equation as
\begin{equation*}
\begin{aligned}& \left(
                   \begin{array}{c}
                    p\\
                     q\\
                   \end{array}
                 \right)
'+\left(
    \begin{array}{cc}
      0& \omega^2\\
       -1 & 0 \\
    \end{array}
  \right)\left(
                   \begin{array}{c}
                     p \\
                      q \\
                   \end{array}
                 \right)=
\left(
                                                                           \begin{array}{c}
                                                                          k^2(2q^3-q)\\
0
                                                                           \end{array}
                                                                         \right).
\end{aligned}\end{equation*}
 It is a Hamiltonian system with the Hamiltonian
$$H(p,q)=\frac{1}{2}p^2+\frac{1}{2}\omega^2q^2+\frac{k^2}{2}(q^2-q^4).$$
The exact solution of this problem is
$$q(t)=sn(\omega t; k/\omega),$$
where $sn$ is the Jacobian elliptic function.

Let $\omega=30$, $k=0.01$. This problem is solved on the interval $[0,100]$ with the stepsize $h=1/30$, Fig. \ref{DFfirst} (b) and
Fig. \ref{DFenergy} show the energy preservation behaviour for IMSVERK1s1, EEuler and IMEEuler.  We also integrate the system over the interval $[0,10]$ with stepsizes $h=1/2^k,\ k=4,\ldots,8$ for IMSVERK1s1, EEuler, IMEEuler, IMMVERK12, IMSVERK12, IMERK12, IMMVERK24,
IMSVERK24, IMERK24, which are shown in Figs \ref{DFfirst} (a),  \ref{DFsecond} and \ref{DFfourth}.

\begin{figure}[!htb]
\centering
\begin{tabular}[c]{cccc}%
  \subfigure[]{\includegraphics[width=6cm,height=5.7cm]{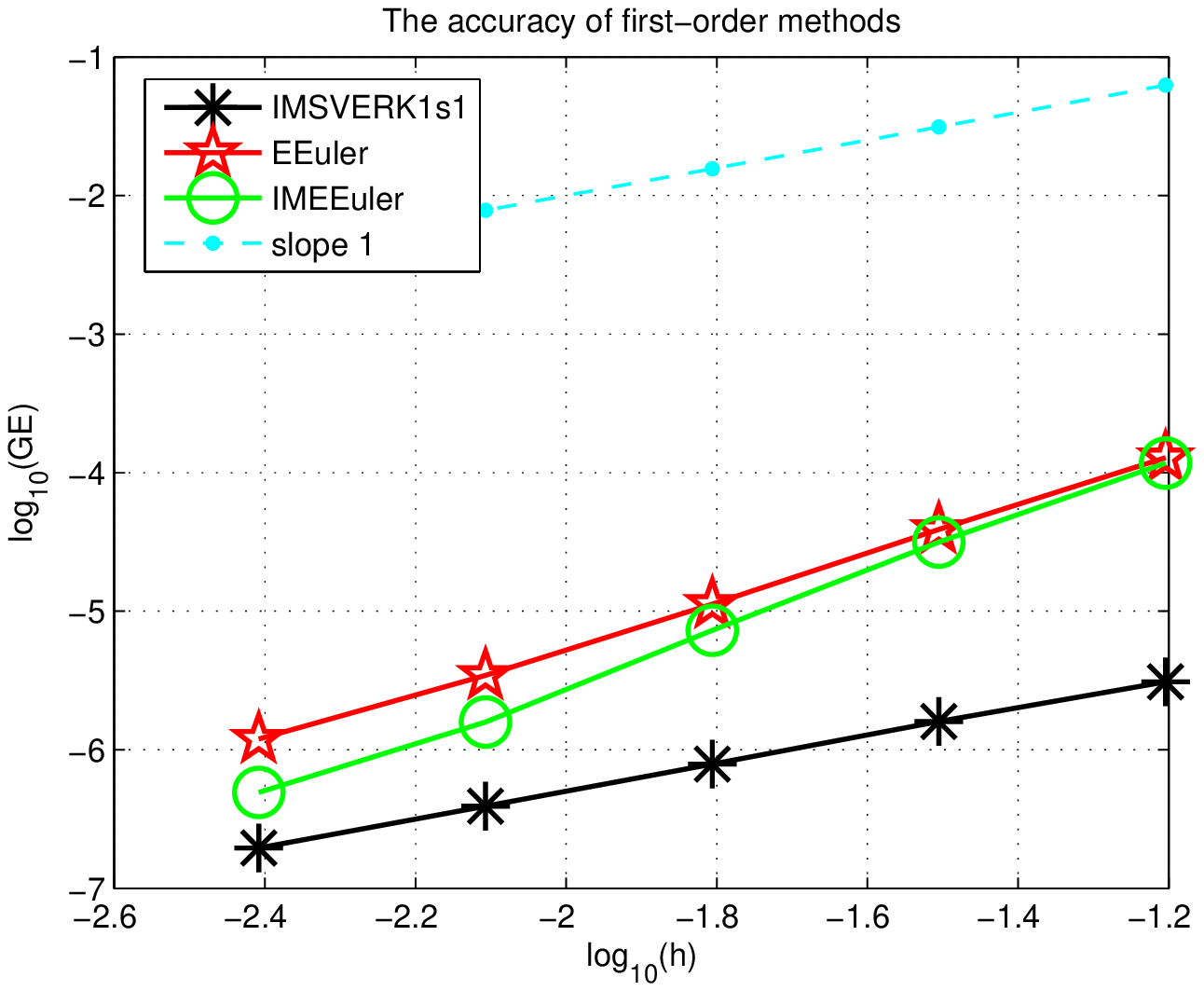}}
  \subfigure[]{\includegraphics[width=6cm,height=5.7cm]{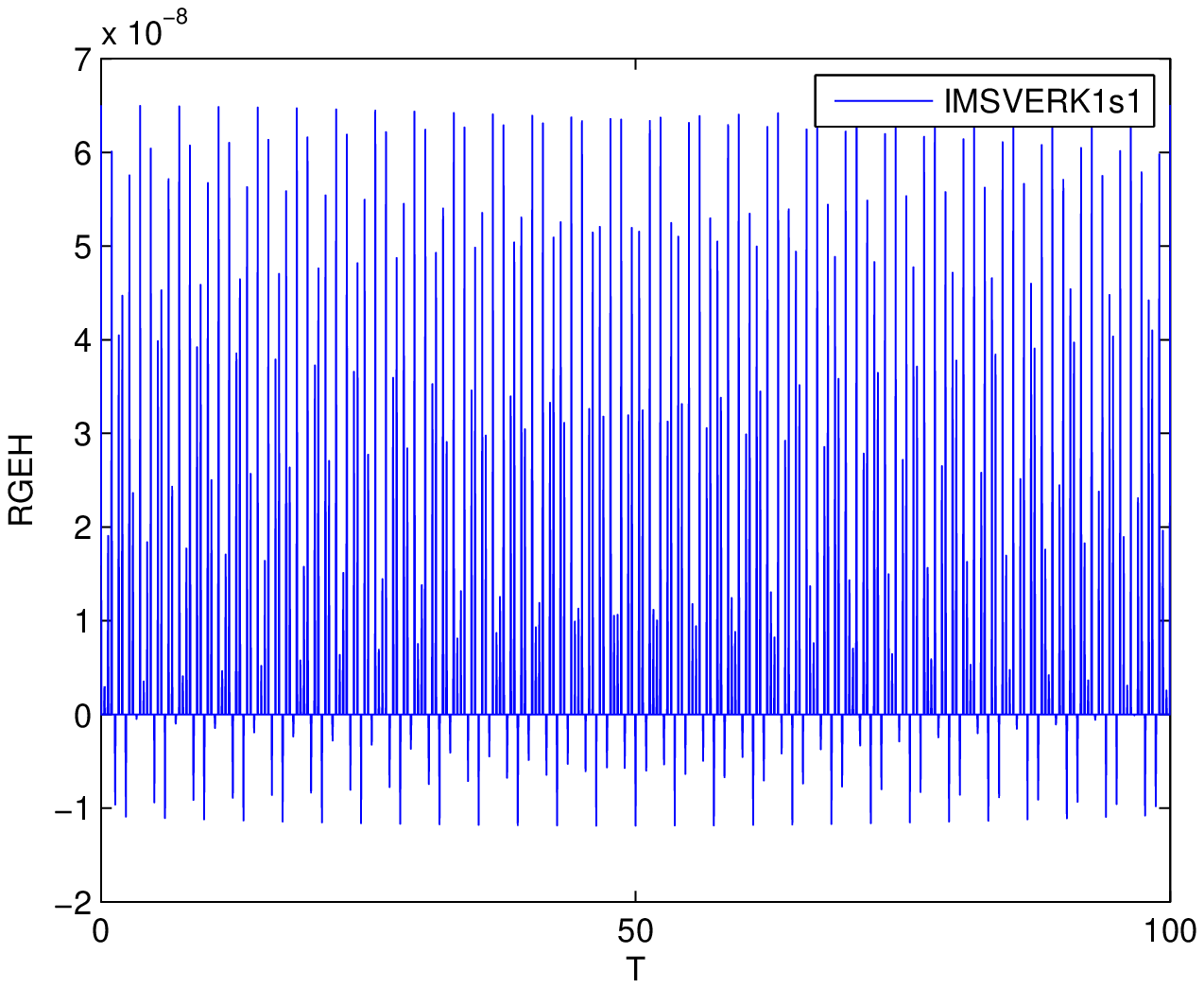}}
\end{tabular}
\caption{Results for  Problem \ref{Duffing}. {\bf {(a)}}: The $\log$-$\log$ plots of global
errors (GE) against $h$. {\bf {(b)}}: the energy preservation for method IMSVERK1s1.}\label{DFfirst}
\end{figure}

\begin{figure}[!htb]
\centering
\begin{tabular}[c]{cccc}%
  \subfigure[]{\includegraphics[width=6cm,height=5.7cm]{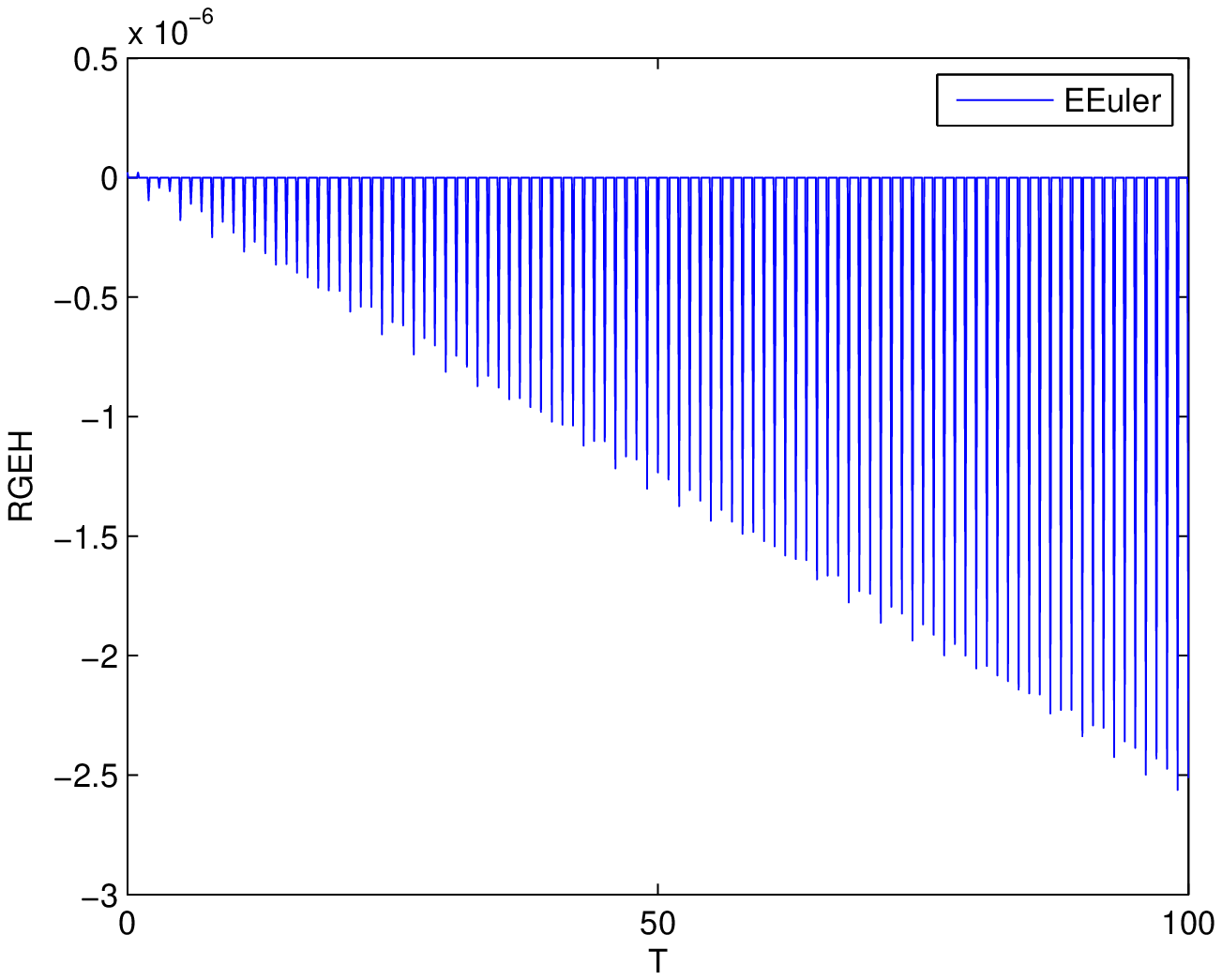}}
  \subfigure[]{\includegraphics[width=6cm,height=5.7cm]{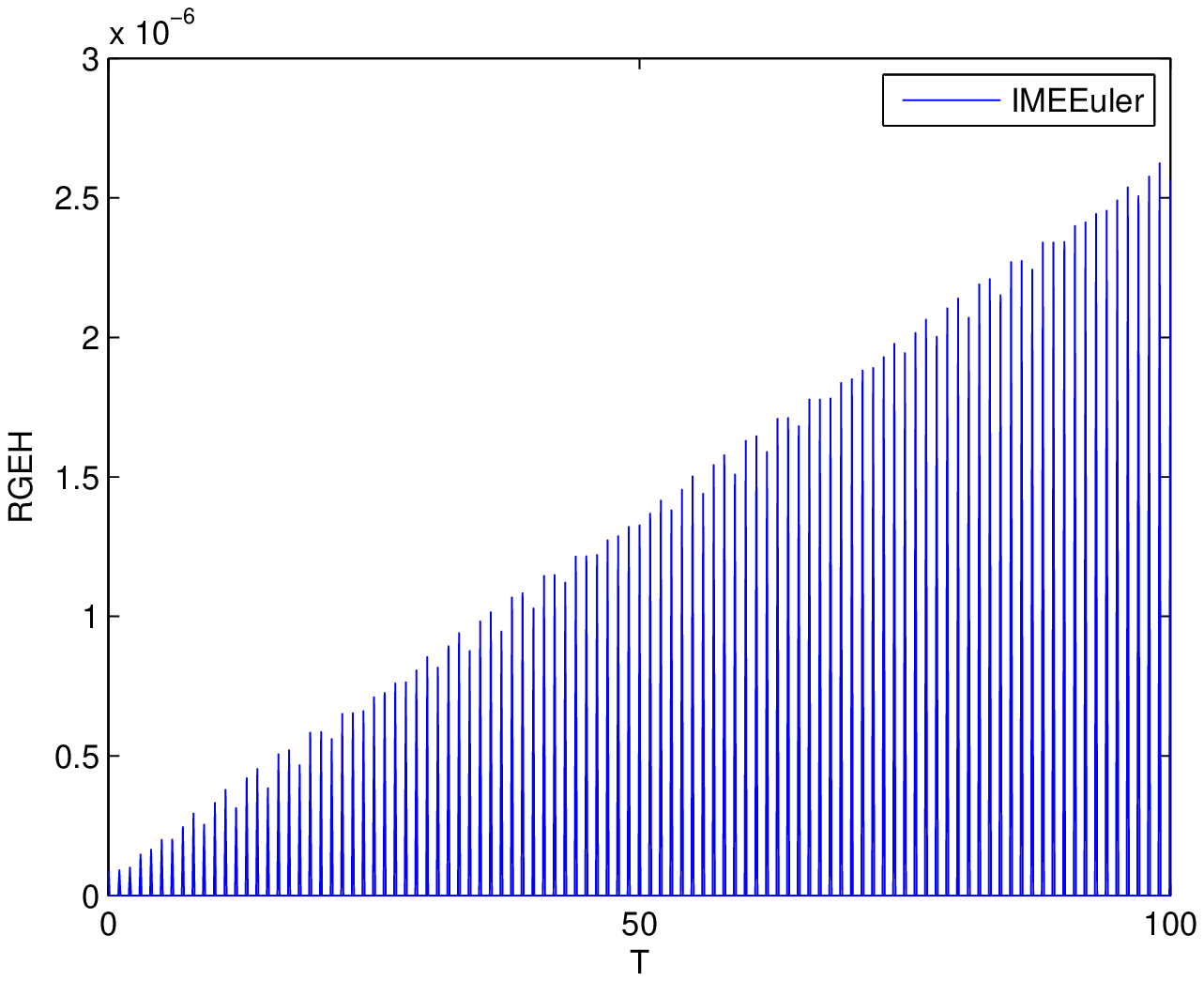}}
\end{tabular}
\caption{Results for  Problem \ref{Duffing}. : the energy preservation for methods EEuler {\bf {(a)}} and IMEEuler {\bf {(b)}}.}\label{DFenergy}
\end{figure}

\begin{figure}[!htb]
\centering
\begin{tabular}[c]{cccc}%
  \subfigure[]{\includegraphics[width=6cm,height=5.4cm]{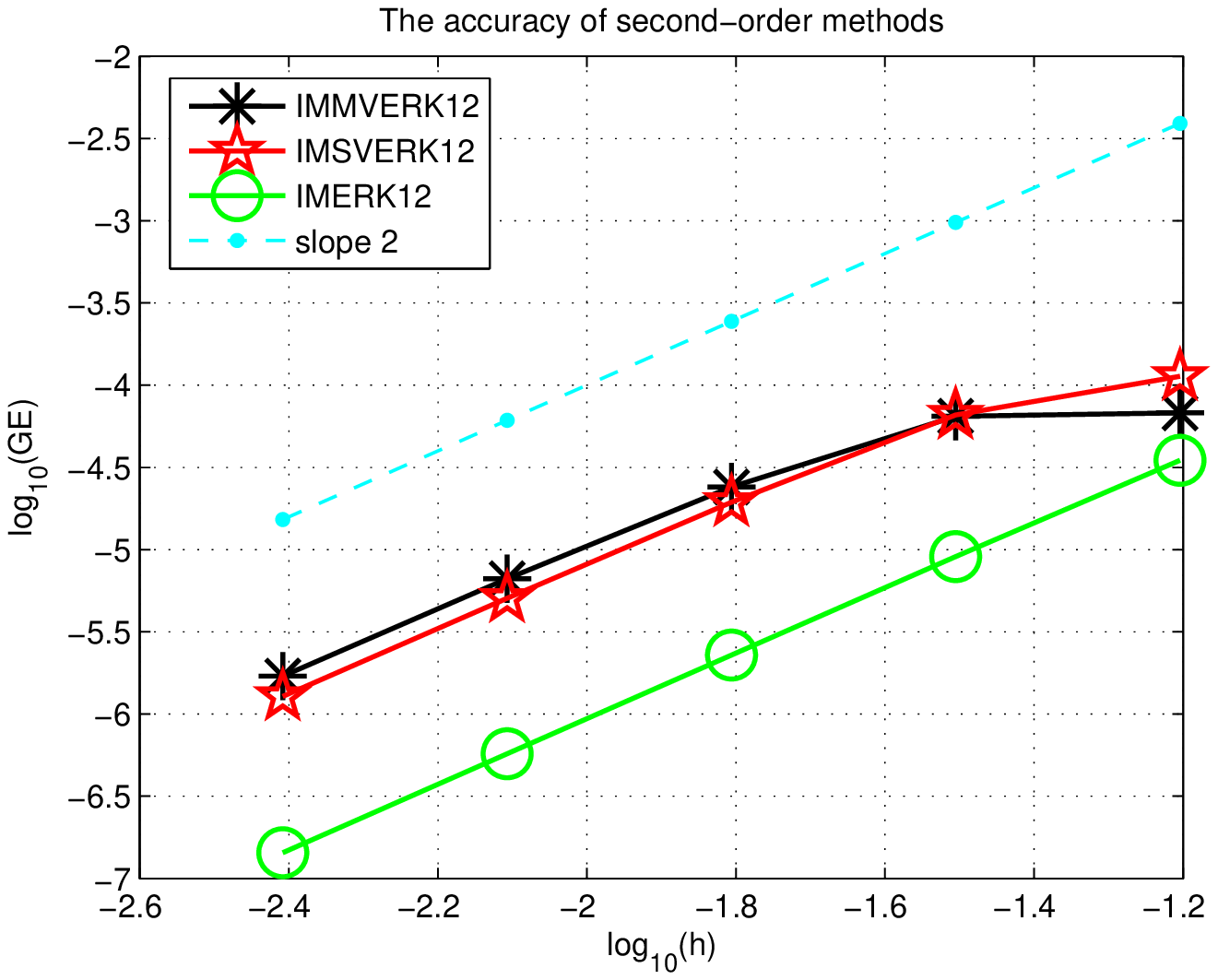}}
  \subfigure[]{\includegraphics[width=6cm,height=5.4cm]{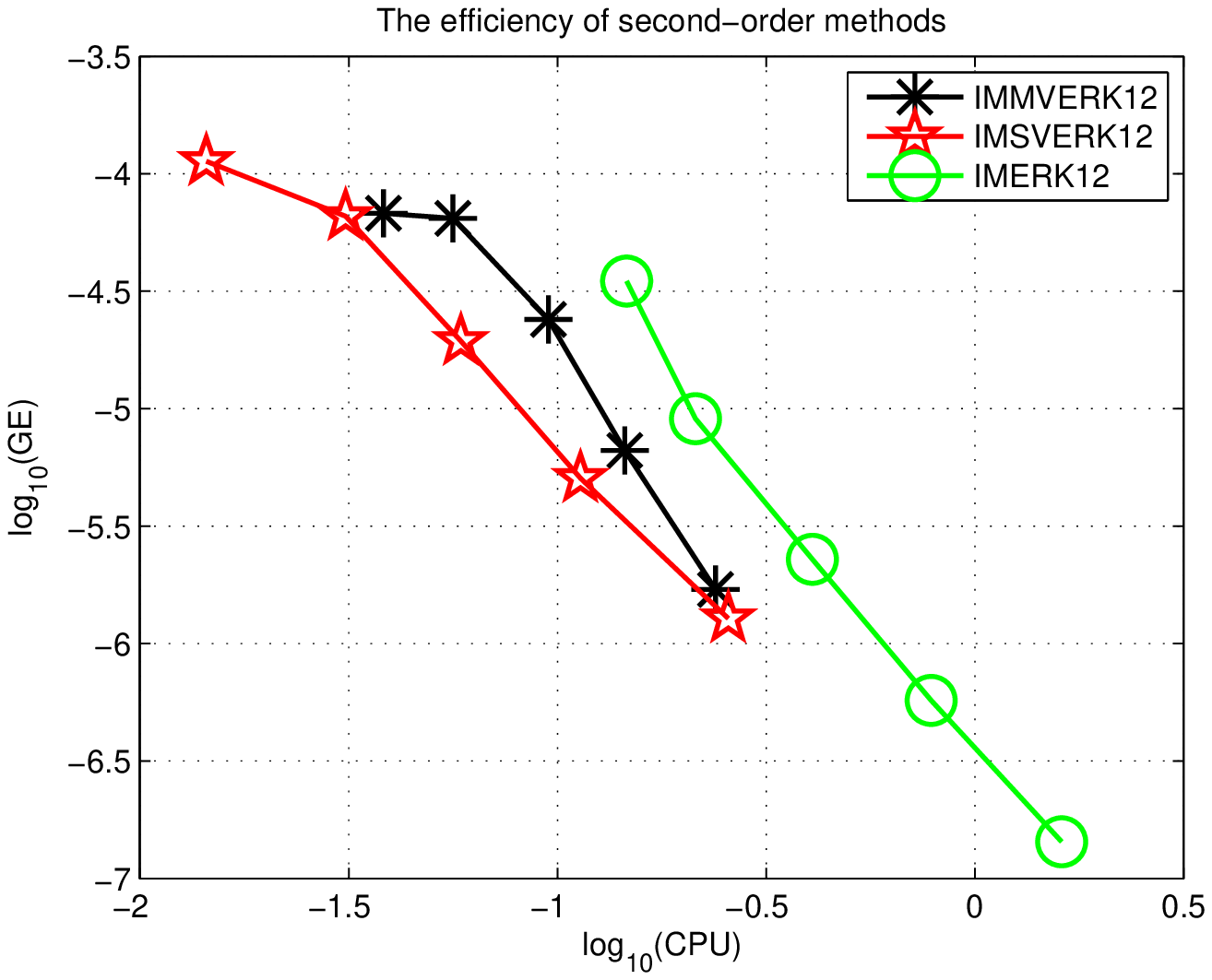}}
\end{tabular}
\caption{Results for  Problem \ref{Duffing}. {\bf {(a)}}: The $\log$-$\log$ plots of global
errors (GE) against $h$. {\bf {(b)}}: The $\log$-$\log$ plots of global
errors against the CPU time.}\label{DFsecond}
\end{figure}

\begin{figure}[!htb]
\centering
\begin{tabular}[c]{cccc}%
  \subfigure[]{\includegraphics[width=6cm,height=5.4cm]{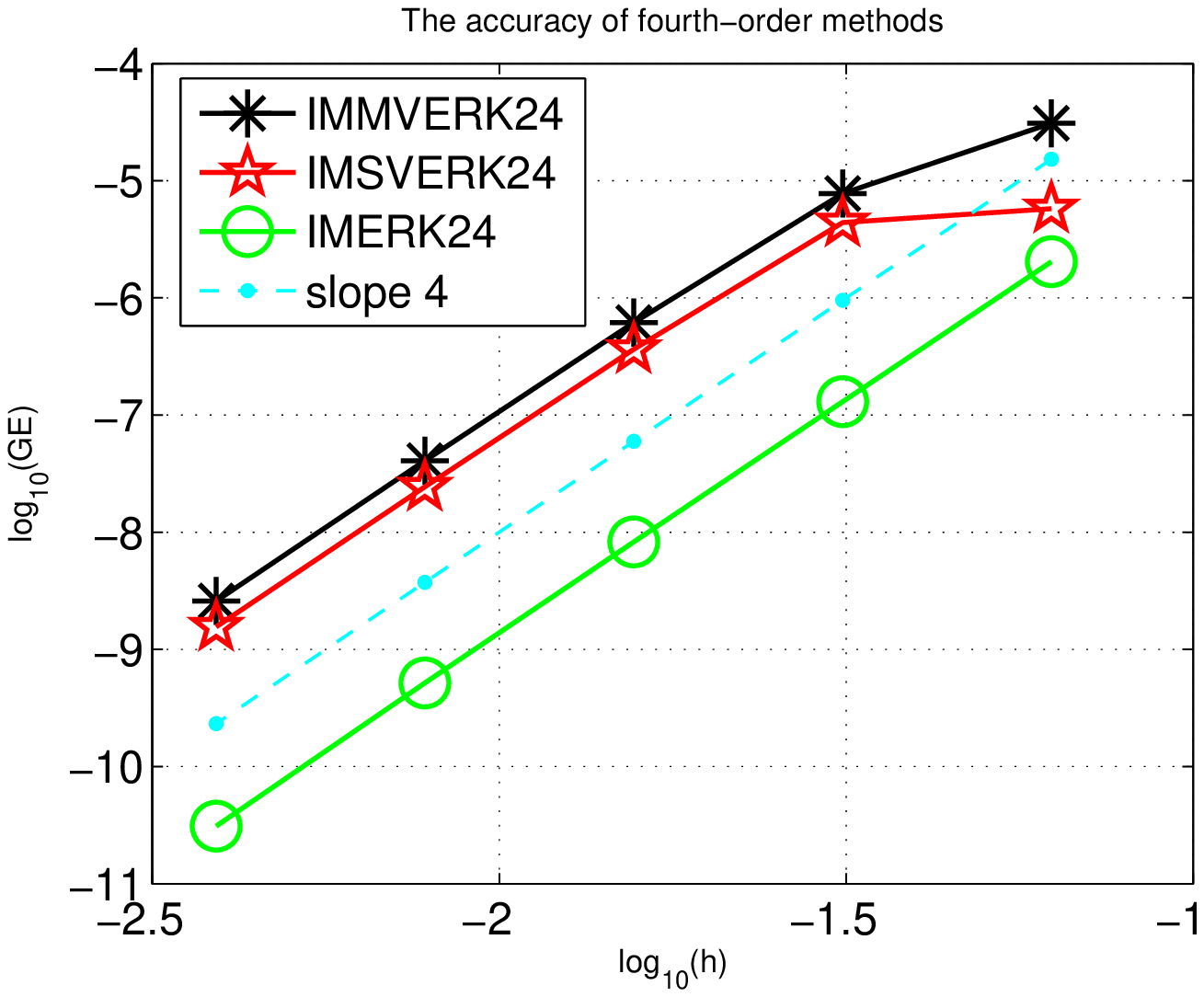}}
  \subfigure[]{\includegraphics[width=6cm,height=5.4cm]{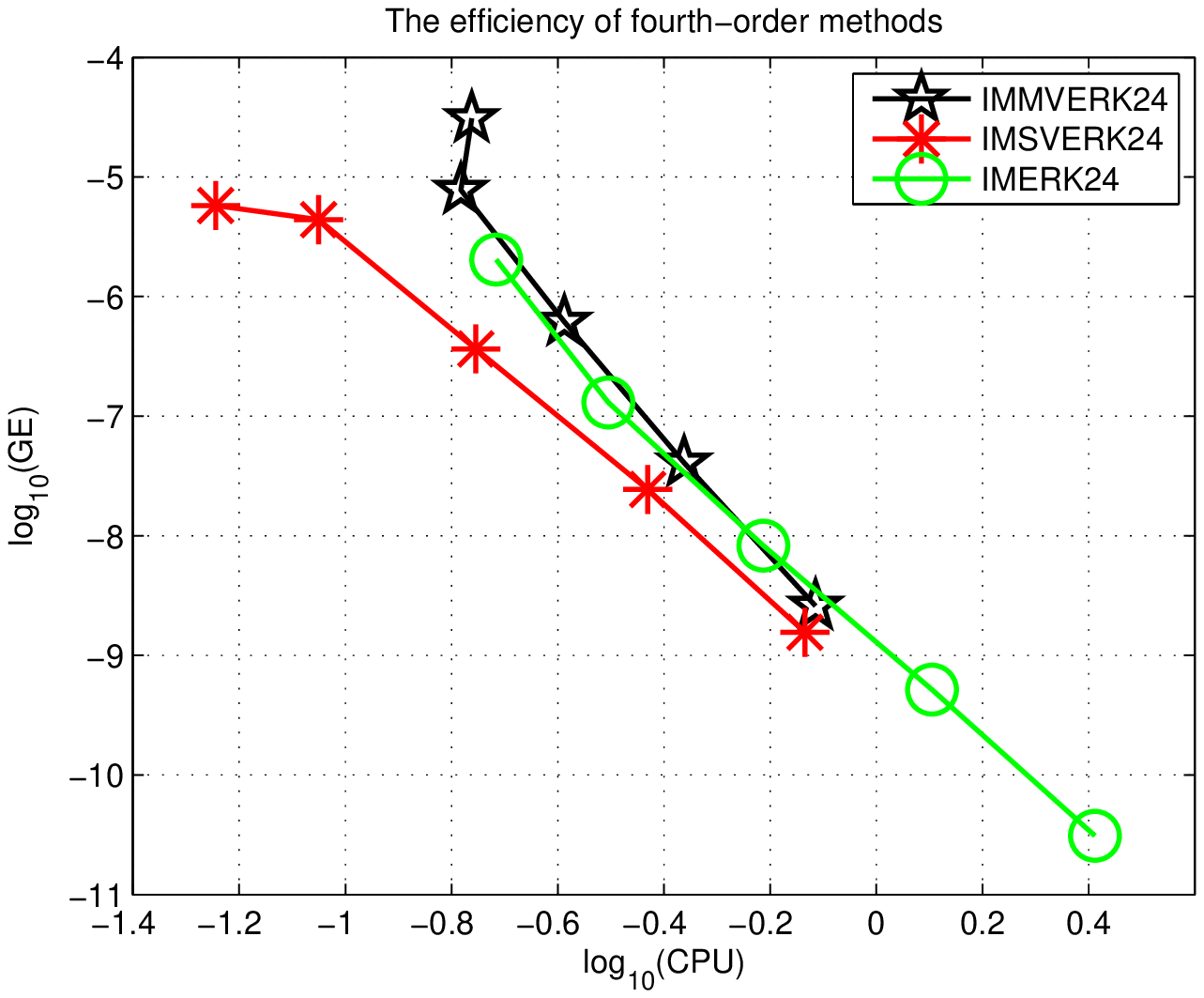}}
\end{tabular}
\caption{Results for  Problem \ref{Duffing}. {\bf {(a)}}: The $\log$-$\log$ plots of global
errors (GE) against $h$. {\bf {(b)}}: The $\log$-$\log$ plots of global
errors against the CPU time.}\label{DFfourth}
\end{figure}

\begin{problem} \label{Sine-Gordon}  Consider the sine-Gorden equation with
periodic boundary conditions \cite{Mei2017}
\[\left\{\begin{array}{l}
\dfrac{\partial^2u}{\partial t^2}=\dfrac{\partial^2u}{\partial
x^2}-\sin(u), \quad -1<x<1, \ t>0,\cr\noalign{\vskip1truemm}
u(-1,t)=u(1,t).\cr\noalign{\vskip1truemm}
\end{array}\right.\]
Discretising the spatial derivative $\partial_{xx}$ by the second-order symmetric differences, which leads to the  following Hamiltonian system
\[\frac{d}{dt}\left(\begin{array}[c]{c}
U^{\prime}\\
U\\
\end{array}\right)+\left(\begin{array}[c]{cc}
{\bf 0}&M\\
-I&{\bf 0}\\
\end{array}\right)\left(\begin{array}[c]{c}
U^{\prime}\\
U\\
\end{array}\right)=\left(\begin{array}[c]{c}
-\sin(U)\\
{\bf 0}\\
\end{array}\right), \quad t\in[0,t_{\rm end}],
\]
whose Hamiltonian is shown by
\begin{equation*}
H(U^{\prime},U)=\frac{1}{2}{U^{\prime}}^{\intercal}U^{\prime}+\frac{1}{2}U^{\intercal}MU^{\intercal}-(\cos u_1+\cdots+u_N).
\end{equation*}
In here, $U(t)=(u_1(t),\cdots,u_N(t))^T$ with $u_i(t)\approx u(x_i,t)$
{for $i=1,\ldots,N$,}  with $\Delta x=2/N$ and $x_i=-1+i\Delta x$,
$F(t,U)=-\sin(u)=-(\sin(u_1),\cdots,\sin(u_N))^T$, and
\[M=\dfrac{1}{\Delta x^2}
\left(
\begin{array}
[c]{ccccc}
2 & -1 &  &  &-1 \\
-1 & 2&  -1 &  &  \\
& \ddots & \ddots & \ddots &   \\
  &  & -1 & 2 & -1\\
  -1&  &   &  -1&2  \\
\end{array}
\right).
\]

In this test, we choose the initial value conditions
\[U(0)=(\pi)^N_{i=1}, \ U'(0)=\sqrt{N}\left(0.01+\sin(\dfrac{2\pi i}{N})\right)_{i=1}^N\]
with $N=48$, and solve the problem on the interval $[0,1]$ with stepsizes $h=1/2^k,\ k=4,\ldots,8$.
The global errors GE against the stepsizes and the CPU time (seconds)
for IMSVERK1s1, EEuler, IMEEuler, IMMVERK12, IMSVERK12, IMERK12, IMMVERK24,
IMSVERK24, IMERK24, which are respectively presented in  Figs \ref{SGfirst} (a),  \ref{SGsecond} and \ref{SGfourth}. Then we integrate this problem on the interval $[0,100]$ with stepsize $h=1/40$,   the energy preservation behaviour for IMSVERK1s1, EEuler, IMEEuler  are shown in Fig. \ref{SGfirst} (b) and Fig. \ref{SGenergy}.
\end{problem}

\begin{figure}[!htb]
\centering
\begin{tabular}[c]{cccc}%
  \subfigure[]{\includegraphics[width=6cm,height=5.4cm]{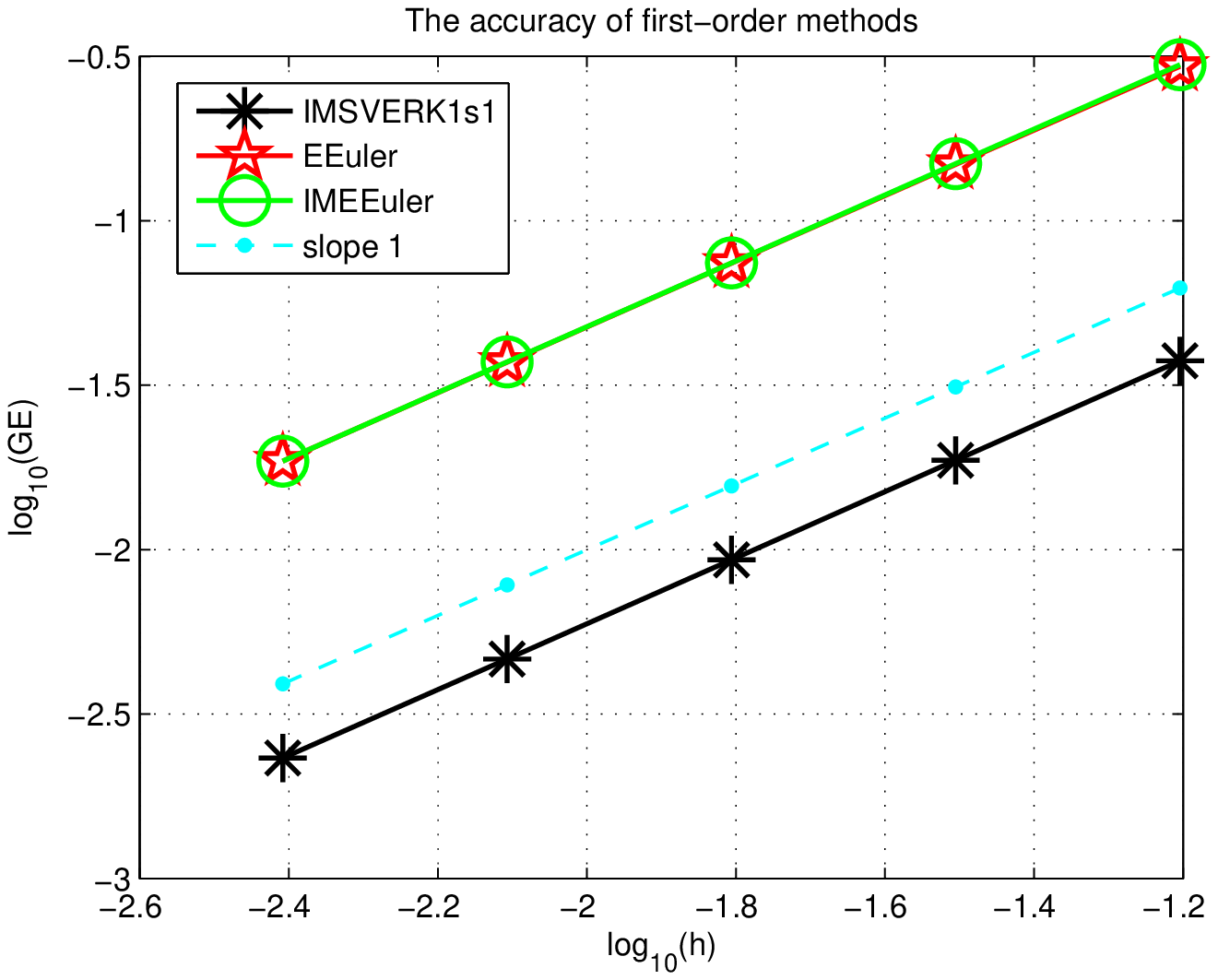}}
  \subfigure[]{\includegraphics[width=6cm,height=5.4cm]{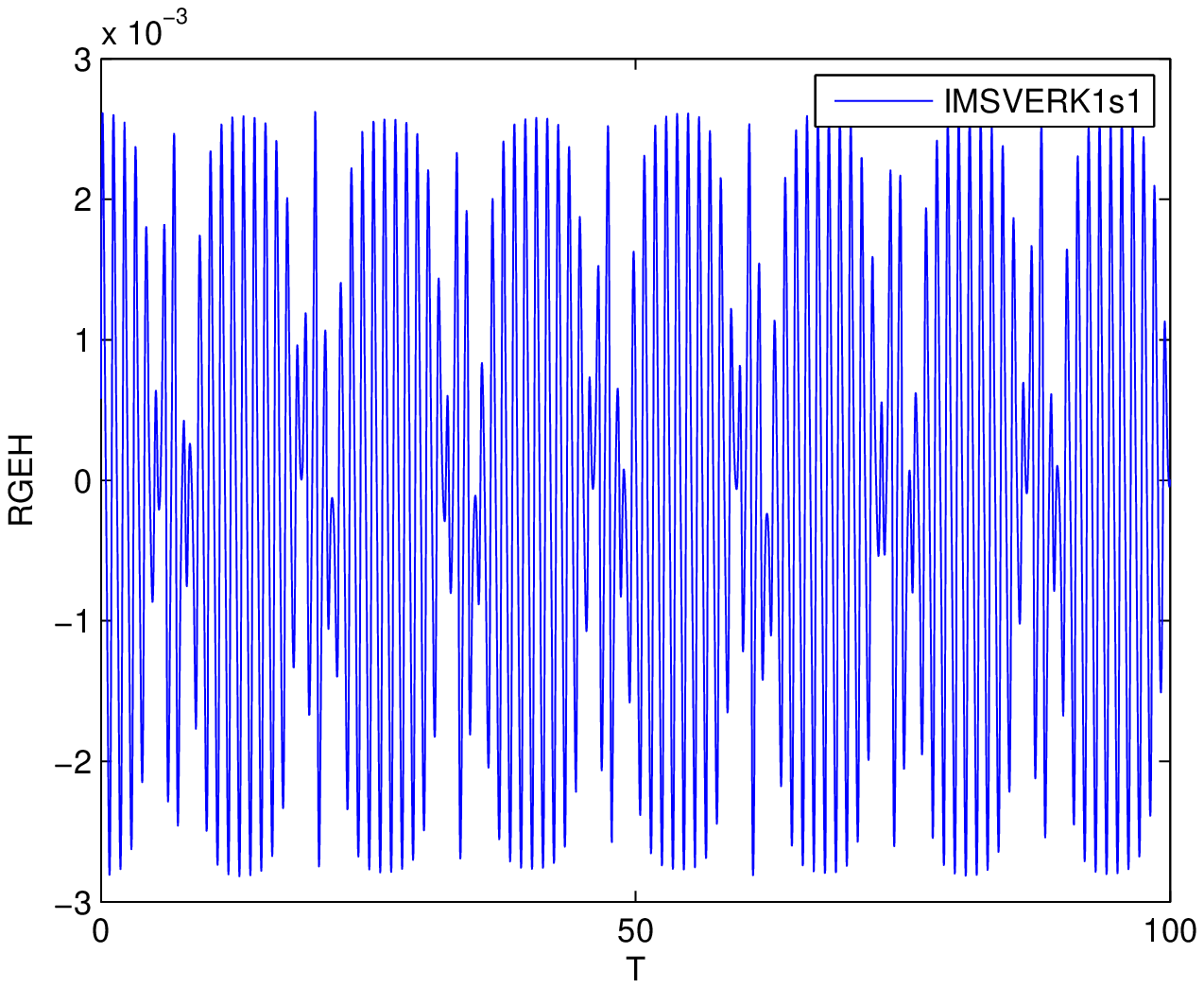}}
\end{tabular}
\caption{Results for  Problem \ref{Sine-Gordon}. {\bf {(a)}}: The $\log$-$\log$ plots of global
errors (GE) against $h$. {\bf {(b)}}: the energy preservation for method IMSVERK1s1.}\label{SGfirst}
\end{figure}

\begin{figure}[!htb]
\centering
\begin{tabular}[c]{cccc}%
  \subfigure[]{\includegraphics[width=6cm,height=5.4cm]{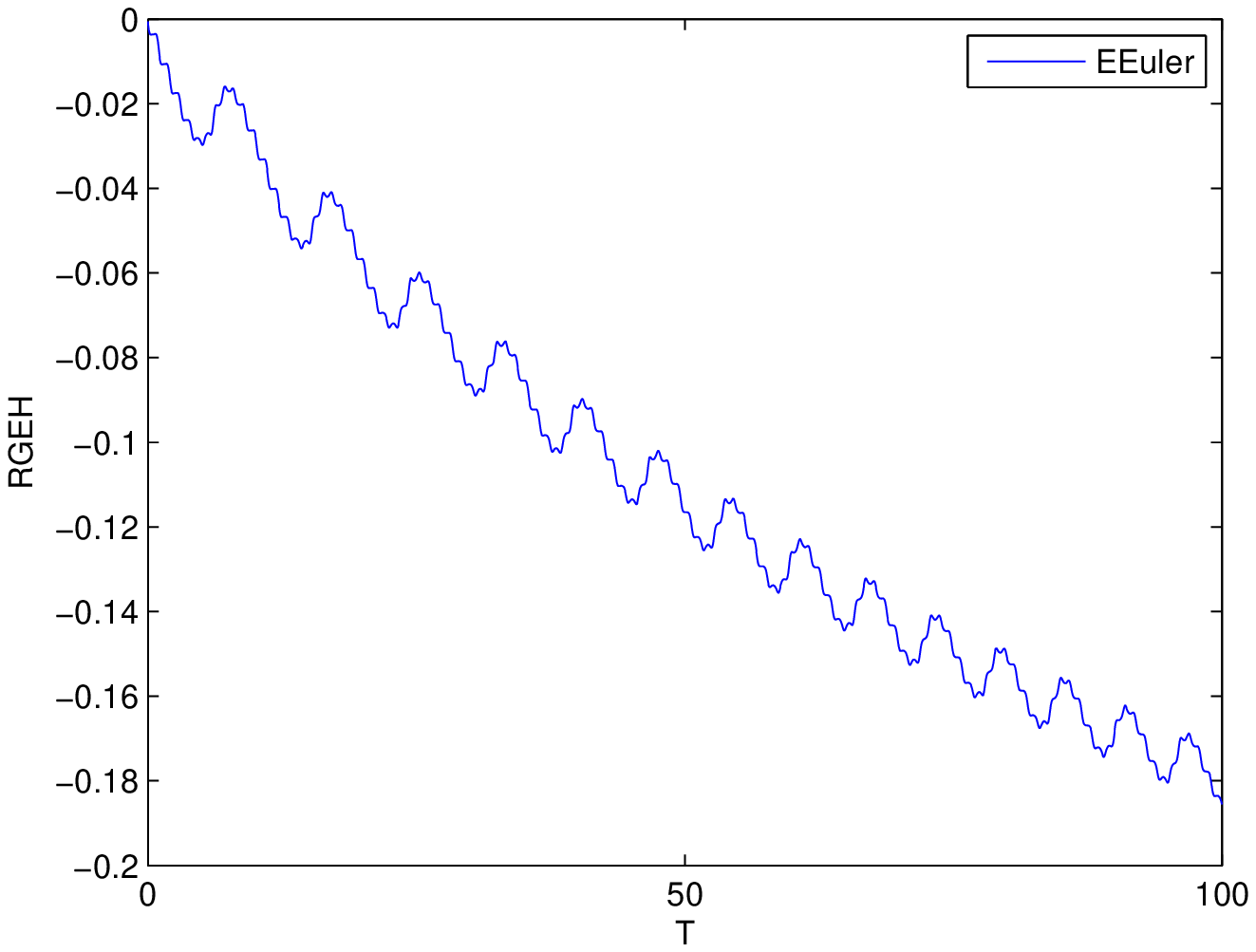}}
  \subfigure[]{\includegraphics[width=6cm,height=5.4cm]{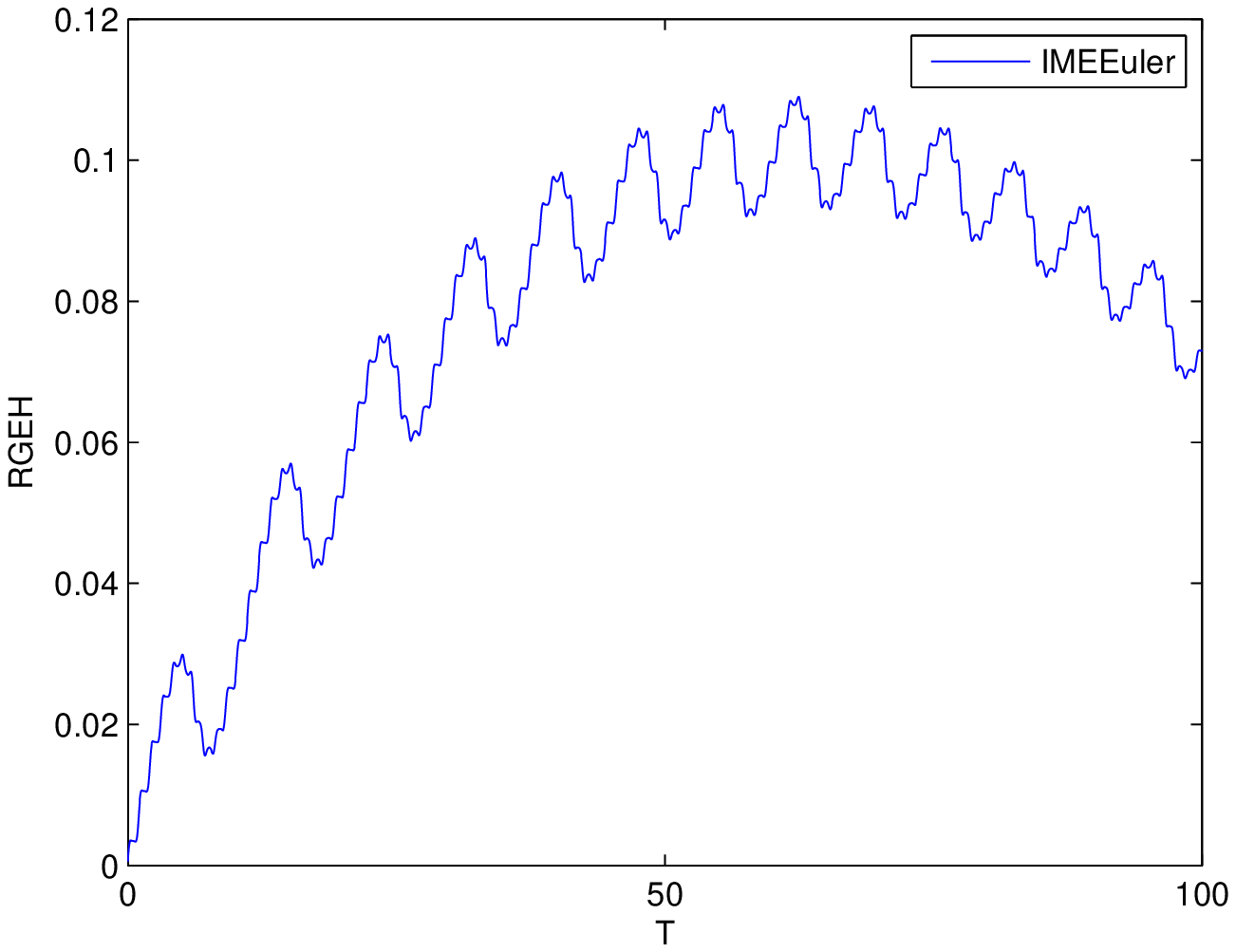}}
\end{tabular}
\caption{Results for  Problem \ref{Sine-Gordon}. : the energy preservation for methods EEuler {\bf {(a)}} and IMEEuler {\bf {(b)}}.}\label{SGenergy}
\end{figure}

\begin{figure}[!htb]
\centering
\begin{tabular}[c]{cccc}%
  \subfigure[]{\includegraphics[width=6cm,height=5.4cm]{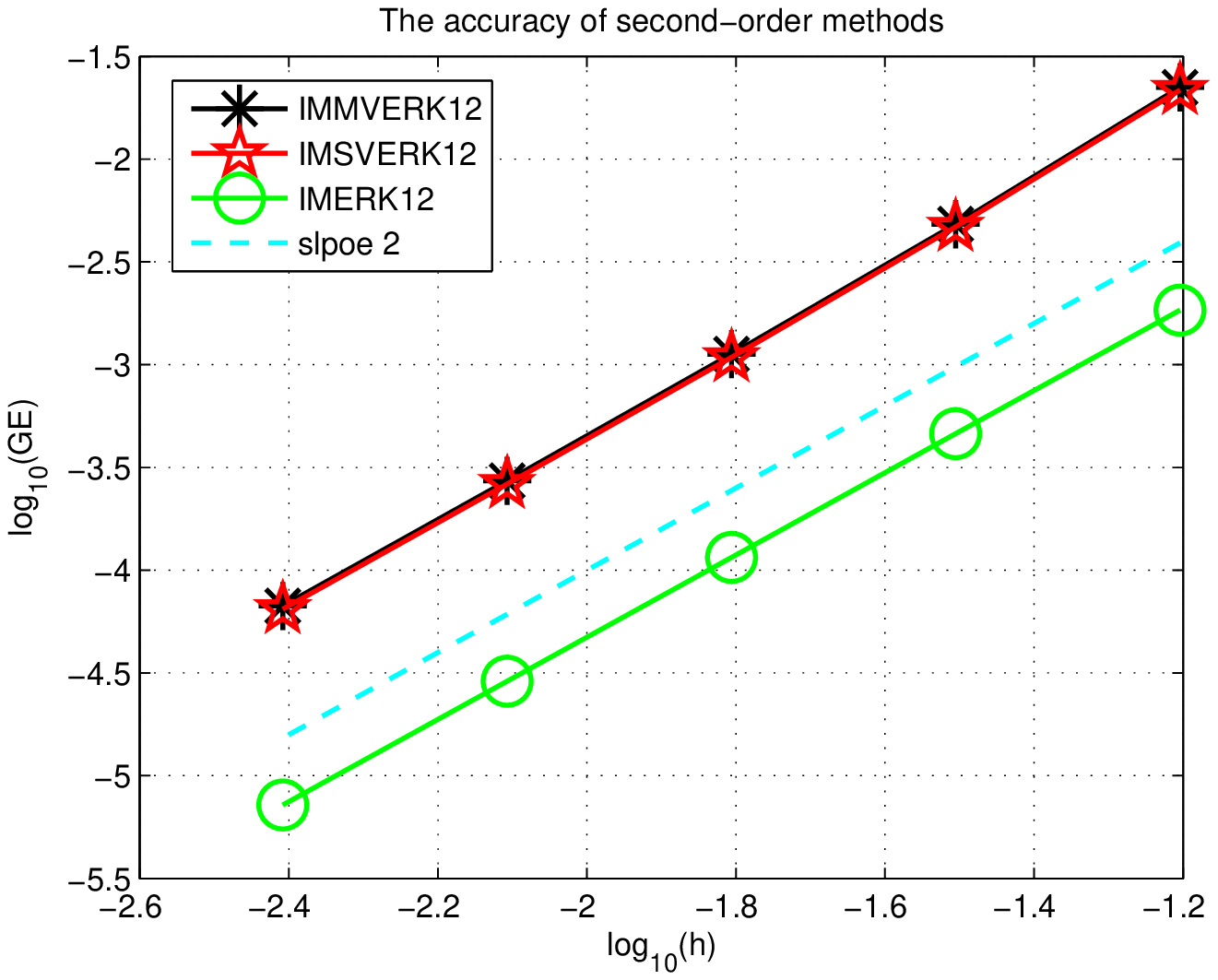}}
  \subfigure[]{\includegraphics[width=6cm,height=5.4cm]{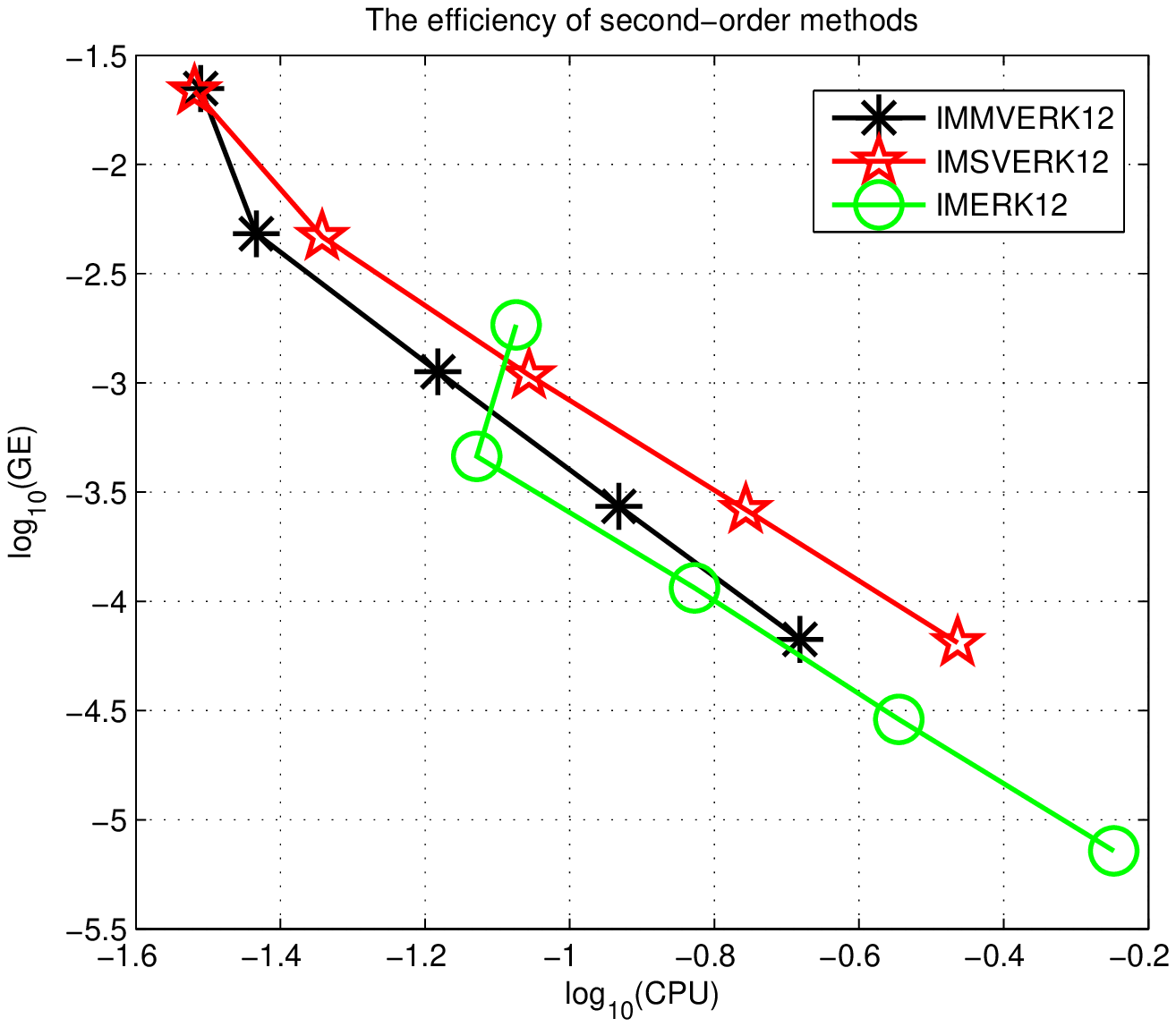}}
\end{tabular}
\caption{Results for  Problem \ref{Sine-Gordon}. {\bf {(a)}}: The $\log$-$\log$ plots of global
errors (GE) against $h$. {\bf {(b)}}: The $\log$-$\log$ plots of global
errors against the CPU time.}\label{SGsecond}
\end{figure}

\begin{figure}[!htb]
\centering
\begin{tabular}[c]{cccc}%
  \subfigure[]{\includegraphics[width=6cm,height=5.4cm]{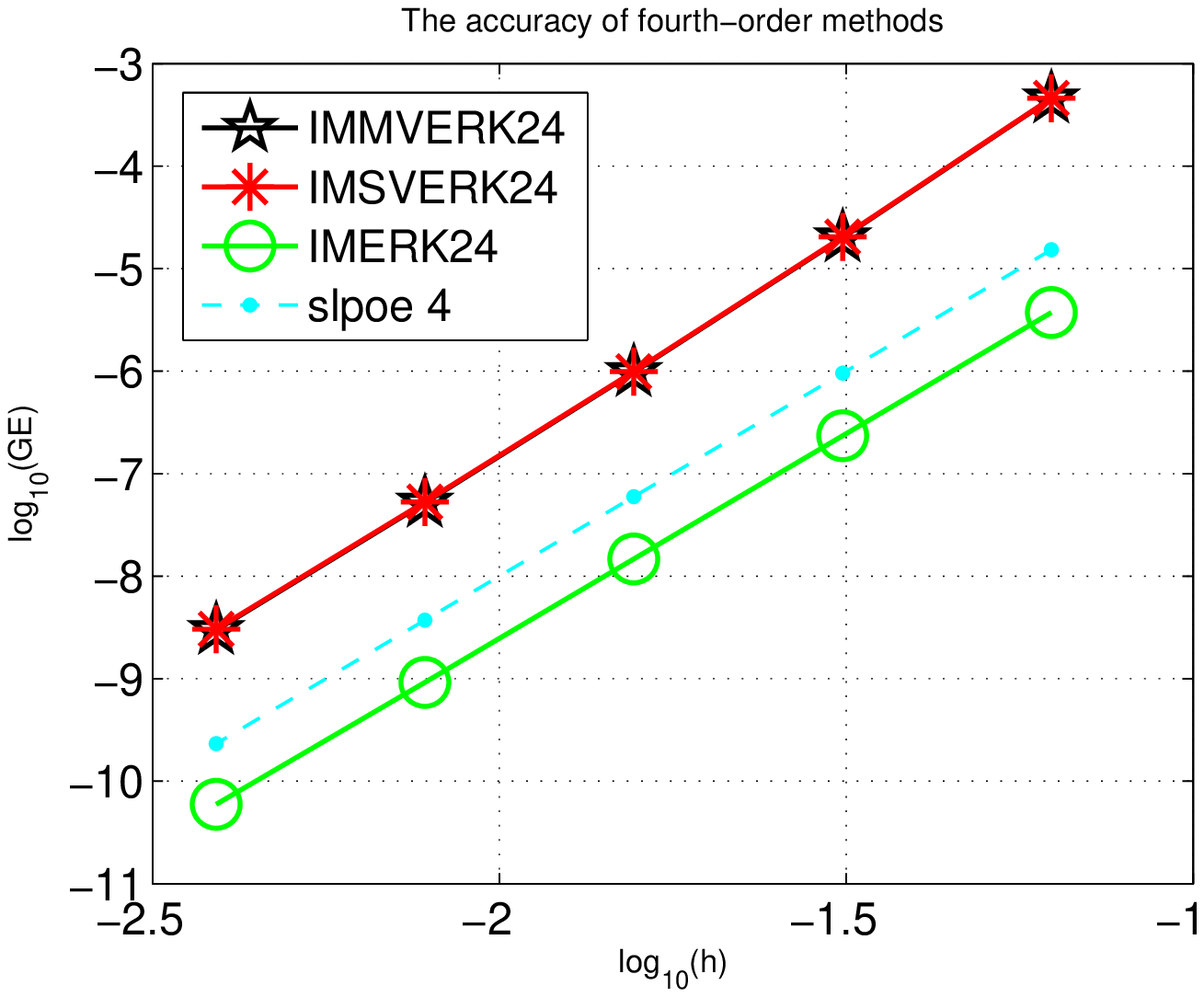}}
  \subfigure[]{\includegraphics[width=6cm,height=5.4cm]{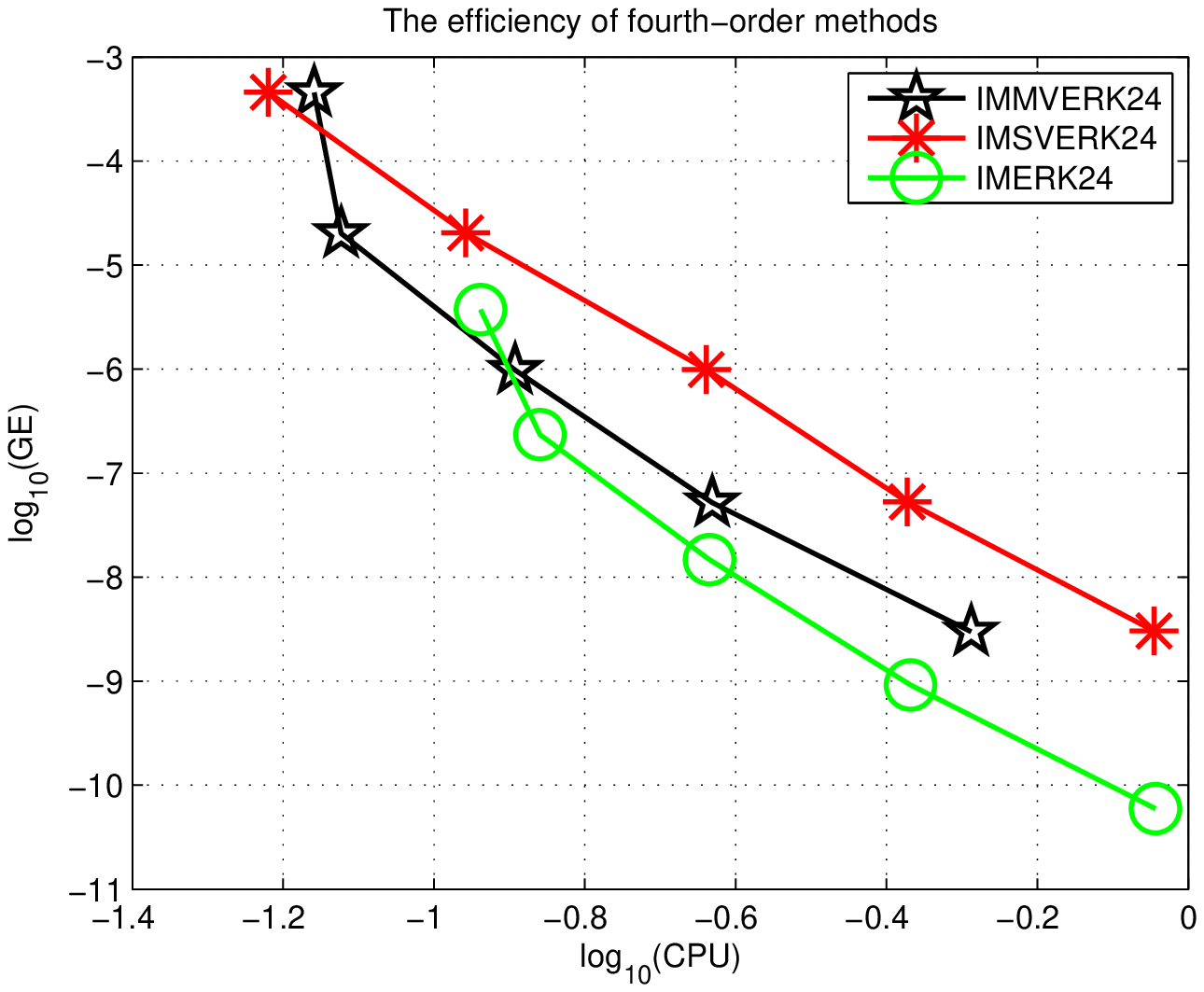}}
\end{tabular}
\caption{Results for  Problem \ref{Sine-Gordon}. {\bf {(a)}}: The $\log$-$\log$ plots of global
errors (GE) against $h$. {\bf {(b)}}: The $\log$-$\log$ plots of global
errors against the CPU time.}\label{SGfourth}
\end{figure}

\section{Conclusion}\label{sec6}

Exponential Runge--Kutta methods  have the unique advantage for solving highly oscillatory problems,  however the implementation of ERK methods generally depends on the evaluations of matrix exponentials.  To reduce computational cost, two new classes of explicit ERK integrators were formulated in \cite{Hu2022,Wang2022}.   Firstly,  we analyzed  the symplectic conditions and verified the existence of the symplectic method, however, the symplectic method only had order one.  Then we designed some practical and effective numerical methods,  and  the order conditions of these ERK methods were exactly identical to standard RK methods. Furthermore,  the linear stability regions for  implicit ERK methods were investigated. Numerical results not only presented the  energy preservation behaviour  for IMSVERK1s1, but also demonstrated the comparable accuracy and efficiency for IMSVERK1s1, IMMVERK12, IMMVERK24, IMSVERK12, IMSVERK24, when applied to the
H\'{e}non-Heiles Model, the Duffing equation and the sine-Gordon equation. It should be noted that our study is based on the classical order conditions, therefore  order reduction phenomenon of these methods can be  observed by several numerical examples.

 Exponential integrators show the better performance than non-exponential integrators. High accuracy and structure preservation for exponential integrators can be further investigated.

\section{Declarations}

\subsection*{Acknowledgements}
The authors are very grateful to the editor and anonymous referees for their invaluable comments and suggestions which helped to improve the manuscript.
\subsection*{Funding}
 This research is partially supported by the National Natural Science Foundation of China (Nos. 12071419).
\subsection*{Conflicts of interest}
The authors declare there is no conflicts of interest regarding the publication of this paper.

{}

\end{document}